\documentclass[a4paper,12pt]{amsart}
\usepackage[utf8]{inputenc}

\usepackage{url}
\usepackage[margin=1in]{geometry}  
\usepackage{epsfig}
\usepackage{color}
\usepackage{graphicx}              
\usepackage{amsmath}
\usepackage{amscd}               
\usepackage{amsfonts}
\usepackage{amssymb}             
\usepackage{amsthm}                
\usepackage{epsfig}
\usepackage{mathrsfs}
\usepackage{hyperref}
\usepackage{enumerate}
\usepackage{esint}

\newtheorem{theorem}{Theorem}

\newtheorem*{theorem*}{Theorem}
\newtheorem{lemma}[theorem]{Lemma}
\newtheorem*{lemma*}{Lemma}

\newtheorem*{proposition*}{Proposition}
\theoremstyle{definition}

\theoremstyle{remark}

\newcommand{\Disc}{\operatorname{Disc}}
\newcommand{\RE}{\operatorname{Re}}
\newcommand{\IM}{\operatorname{Im}}
\newcommand{\vol}{\operatorname{vol}}

\newcommand{\sgn}{\operatorname{sgn}}

\newcommand{\cusp}{\mathrm{cusp}}
\newcommand{\Eis}{\mathrm{Eis}}

\newcommand{\bC}{\mathbb{C}}
\newcommand{\bH}{\mathbb{H}}

\newcommand{\bQ}{\mathbb{Q}}
\newcommand{\bR}{\mathbb{R}}

\newcommand{\zed}{\mathbb{Z}}

\newcommand{\GO}{\mathrm{GO}}
\newcommand{\GL}{\mathrm{GL}}
\newcommand{\PSL}{\mathrm{PSL}}
\newcommand{\SL}{\mathrm{SL}}
\newcommand{\SO}{\mathrm{SO}}

\newcommand{\sH}{{\mathscr{H}}}
\newcommand{\sI}{{\mathscr{I}}}

\newcommand{\sL}{{\mathscr{L}}}

\newcommand{\sO}{{\mathscr{O}}}

\newcommand{\sS}{{\mathscr{S}}}

\newcommand{\fS}{\mathfrak{S}}

\title{The shape of cubic fields}
\author{Robert Hough}

\begin{document}

\begin{abstract}
We use the method of Shintani, as developed by Taniguchi and Thorne, to prove the quantitative equidistribution  of the shape of cubic fields when the fields are ordered by discriminant.  
\end{abstract}

\thanks{Robert Hough's research is supported by NSF Grants DMS-1712682, ``Probabilistic methods in discrete structures and applications," and DMS-1802336, ``Analysis of Discrete Structures and Applications"}

\maketitle

\section{Introduction}
A degree $n$ number field $K/\bQ$ has $r_1$ real and $r_2$ complex embeddings, $n = r_1 + 2r_2$.  Let the real embeddings be $\sigma_1, ..., \sigma_{r_1}: K \to \bR$ and the complex embeddings be $\sigma_{r_1+1}, ..., \sigma_{r_1 + r_2}: K \to \bC$.  The canonical embedding
\begin{equation}
 \sigma(x) = (\sigma_1(x), ..., \sigma_{r_1+r_2}(x)) \in \bR^{r_1} \times \bC^{r_2}
\end{equation}
is an injective ring homomorphism.  Consider $\bC$ to be a 2-dimensional real vector space.  With this identification, the ring of integers $\sO \subset K$ is an $n$-dimensional lattice in $\bR^n$ under the mapping
\begin{equation}
 x \mapsto (\sigma_1(x), ..., \sigma_{r_1}(x), \RE \sigma_{r_1 + 1}(x), \IM \sigma_{r_1 + 1}(x), ..., \RE \sigma_{r_1 + r_2}(x), \IM \sigma_{r_1 + r_2}(x))
\end{equation}
with covolume
\begin{equation}
 \vol(\sigma(\sO)) = 2^{-r_2}|D|^{\frac{1}{2}}
\end{equation}
where $D=D(K)$ is the field discriminant. 

An old theorem of Hermite \cite{S13} states that there are only finitely many number fields of a given discriminant.  Thus it is natural to ask, when number fields of a fixed degree are ordered by growing discriminant,  how is the ring of integers distributed as a lattice?  Note that, as $\sigma(1)$ is present in the embedding for all $K$, $\sigma(\sO)$ always has a short vector in this direction, relative to the volume.  Thus define the lattice shape $\Lambda_K$ to be the $(n-1)$ dimensional orthogonal projection in the space orthogonal to $\sigma(1)$, rescaled to have covolume 1.  

In the case of $S_n$ fields of degree $n = 3$ and $n = 3, 4, 5$ Terr \cite{T97}  and Bhargava and Harron \cite{BH16} prove that the shape of $\Lambda_K$ becomes equidistributed in the space
\begin{equation}
 \sS_{n-1} := \GL_{n-1}(\zed) \backslash \GL_{n-1}(\bR)/\GO_{n-1}(\bR)
\end{equation}
with respect to the induced probability Haar measure.  Their arguments use the geometry of numbers and obtain only the asymptotic equidistribution.  
A natural basis of functions in which to study equidistribution on the larger quotient \begin{equation}\Lambda_{n-1} := \SL_{n-1}(\zed) \backslash \SL_{n-1}(\bR)\end{equation} consist in joint eigenfunctions of the Casimir operator and its $p$-adic analogues, the Hecke operators.  Thus in the case $n=3$ the space $\SL_{2}(\zed) \backslash \SL_{2}(\bR)$ may be decomposed spectrally into the constant function, cusp forms, and  Eisenstein series.
Our main result obtains quantitative cuspidal equidistribution of the shape of cubic fields when the fields are ordered by increasing size of discriminant.

Identify $\Lambda_K$ with a point $x_K$ in the homogeneous space $\SL_2(\zed)\backslash \SL_2(\bR)$ by choosing  a base-point, which is specified explicitly below. 
\begin{theorem}\label{cubic_field_theorem}
 Let $\phi$ be a cuspidal automorphic form on $\SL_2(\zed)\backslash \SL_2(\bR)$, which transforms on the right by a character of $\SO_2(\bR)$ of degree $2k$, and which is an eigenfunction of the Casimir operator and the Hecke operators.  Let $F \in C_c^\infty(\bR^+)$ be a smooth test function.    For any $\epsilon > 0$, as $X \to \infty$,
 \begin{equation}
  N_{3,\pm}(\phi, F,X):=\sum_{[K:\bQ]=3} \phi(\Lambda_K) F\left(\frac{\pm \Disc(K)}{X} \right) \ll_{\phi, \epsilon}  X^{\frac{2}{3} + \epsilon}.
 \end{equation}

\end{theorem}

This bound should be compared to the number of cubic fields with discriminant of size at most $X$, which is of order $X$. Besides the significant cancellation exhibited in our theorem, the advantage of the method is in obtaining the equidistribution of a further angle of the lattice when oriented in $\bR^3$ relative to $\sigma(1)$, which is accomplished by permitting the cusp form to transform on the right by a character of $\SO_2(\bR)$. A further advantage of the method is that it appears to extend to treat the joint cuspidal equidistribution of the shape of quartic fields when paired with the cubic resolvent, extending work of Yukie \cite{Y93}.  We intend to return to this topic in a forthcoming publication.

In \cite{TT13b} Theorem 1.3 a counting result for cubic fields is proved which permits imposing finitely many local specifications, with applications in \cite{MP13}, \cite{CK15}, and \cite{EPW17}.  
We expect the method of Theorem \ref{cubic_field_theorem} can be extended to accommodate finitely many local conditions, but have not done so here.

\subsection{Discussion of method}
Shintani \cite{S72} and Shintani and Sato \cite{SS74} introduce zeta functions enumerating integral orbits in prehomogeneous vector spaces, proving meromorphic continuation and functional equations. Taniguchi and Thorne \cite{TT13a}, \cite{TT13b} use Shintani's zeta function in the case of binary cubic forms together with a sieve to give the best known error terms in the counting function of cubic fields ordered by discriminant. In \cite{H17} the author modified this construction in the case of binary cubic forms, by introducing an automorphic form evaluated at a representative of each orbit.  In proving Theorem \ref{cubic_field_theorem} we combine the construction of \cite{H17} with the methods of \cite{TT13b}, along with an argument related to the approximate functional equation from the theory of $L$-functions.

\subsection*{Notation and conventions}
We use the following conventions regarding groups. $G_\bR = \GL_2(\bR)$, $G^1 = \SL_2(\bR)$, $G^+ = \{g \in \GL_2(\bR): \det g > 0\}$, $G_\zed = \GL_2(\zed)$, $\Gamma = \SL_2(\zed)$.
Following Shintani, 
\begin{align} A &= \left\{\begin{pmatrix} t &\\ & \frac{1}{t}\end{pmatrix}: t \in \bR_{>0}\right\},\; N = \left\{\begin{pmatrix} 1 &0\\ x & 1 \end{pmatrix}: x \in \bR\right\}, \; N' = \left\{\begin{pmatrix} 1 & x \\ 0& 1 \end{pmatrix}: x \in \bR\right\}, \\ \notag K &= \left\{\begin{pmatrix} \cos \theta & \sin \theta\\ -\sin \theta & \cos \theta \end{pmatrix}: \theta \in \bR\right\}\end{align}
with group elements
\begin{align}
 a_t &= \begin{pmatrix} t &\\& \frac{1}{t}\end{pmatrix}, \; n_x = \begin{pmatrix}1&0\\x &1 \end{pmatrix},\; \nu_x = \begin{pmatrix} 1&x\\ 0&1\end{pmatrix},\; 
 k_\theta = \begin{pmatrix} \cos \theta & \sin \theta\\ -\sin \theta & \cos \theta \end{pmatrix}.
\end{align}
The Iwasawa decomposition $G = KAN$ is used. Haar measure is normalized by setting, for $f \in L^1(G^1)$,
\begin{align}
 \int_{G^1} f(g) dg &= \frac{1}{2\pi} \int_0^{2\pi}\int_{-\infty}^\infty \int_0^\infty f(k_\theta a_t n_u)\frac{dt}{t^3}du d\theta
\end{align}
and for $f \in L^1(G^+)$,
\begin{equation}
 \int_{G^+}f(g)dg = \int_0^\infty \int_{G^1} f\left(\begin{pmatrix} \lambda &\\ &\lambda\end{pmatrix} g \right)dg \frac{d\lambda}{\lambda}.
\end{equation}

We abbreviate contour integrals $\frac{1}{2\pi i} \int_{c-i\infty}^{c + i \infty} F(z) dz = \oint_{\RE(z) = c} F(z)dz$. Denote $e(x)$ the additive character $e^{2\pi i x}$.
The argument uses the following pair of standard Mellin transforms.  Write $K_\nu$ for the $K$-Bessel function.
For $\RE(s)> |\RE \nu|$, (\cite{I02}, p.205)
\begin{equation}
 \int_0^\infty K_\nu(x)x^{s-1}dx = 2^{s-2}\Gamma\left(\frac{s + \nu}{2}\right)\Gamma\left(\frac{s-\nu}{2}\right).
\end{equation}
For $0 < \RE(s) < 1$, (\cite{B53}, p.13)
\begin{equation}
 \int_0^\infty e^{ix}x^{s-1}dx = \Gamma(s)e^{i \frac{\pi s}{2}}.
\end{equation}
Given $F$ which is smooth, of compact support on $\bR^\times$, the Mellin transform
\begin{equation}
 \tilde{F}(s) = \int_0^\infty F(x) x^{s-1}dx
\end{equation}
is entire.  The operator $x \frac{d}{dx}$ acts on the Mellin transform by multiplying by $-s$, as may be checked by integrating by parts,
\begin{equation}\label{mellin_operator}
 \widetilde{x \frac{d}{dx} F}(s) = \int_0^\infty F'(x) x^s dx = -s \int_0^\infty F(x)x^{s-1}dx = -s \tilde{F}(s).
\end{equation}

\section{Cubic rings and binary cubic forms}
A cubic ring $R$ over $\zed$ is a free rank 3 $\zed$ module equipped with a ring multiplication.  A set of generators $\langle 1, \omega, \theta \rangle$ for a cubic ring $R$ is said to be \emph{normalized} if their multiplication law satisfies, for some integers $\ell, m, n, a, b, c, d$,
\begin{align}
\omega \theta &= n\\
\notag \omega^2 &= m + b\omega - a \theta\\
\notag \theta^2 &= \ell + d\omega - c\theta, 
\end{align}
see \cite{B04a}.  A basis for $R$ may be reduced to a normalized basis by choosing representatives for $\omega$ and $\theta$ modulo $\zed$. $\GL_2(\zed)$ acts on cubic rings over $\zed$ by forming linear combinations of the generators $\langle \omega, \theta \rangle$, then renormalizing.  Starting with a distinguished basis $\langle 1, \omega, \theta \rangle$ for $\bR^3$ or $\bR \times \bC$, $\GL_2(\bR)$ acts in the same way.  The `shape' of the basis thus corresponds with a point in the space of lattices $\SL_2(\zed)\backslash \SL_2(\bR)$.  This is the usual identification of $\SL_2(\zed)\backslash \SL_2(\bR)$ with the space of 2-dimensional lattices, since normalizing the basis acts in the direction of 1. 

The fact that, after tensoring with $\bR$, each cubic ring over $\zed$ of non-zero discriminant may be realized as a point in one of these two spaces has been proven via a correspondence with the theory of real and integral binary cubic forms.
An \emph{integral binary cubic form} is a form $f(x,y) = ax^3 + bx^2y + cxy^2 + dy^3$ with $(a,b,c,d) \in \zed^4$.  $\GL_2(\zed)$ acts on the space of integral binary cubic forms by forming linear combinations of $x$ and $y$.
Gan, Gross and Savin \cite{GGS02}, extending earlier work of Delone and Fadeev \cite{DF64}, proved the following parameterization of cubic rings over $\zed$. 
\begin{theorem}[\cite{B04a}, Theorem 1]
	There is a canonical bijection between the set of $\GL_2(\zed)$-equivalence classes of integral binary cubic forms and the set of isomorphism classes of cubic rings, in which the form $f(x,y) = ax^3 + bx^2y + cxy^2 + dy^3$ corresponds to the ring $R(f)$ with basis $\langle 1, \omega, \theta\rangle$ and multiplication law
	\begin{align}
	\omega \theta &= -ad\\\notag
	\omega^2 &= -ac + b \omega -a\theta\\
	\notag \theta^2 &= -bd + d\omega - c\theta.
	\end{align}
	Moreover, the discriminant of $f$ and $R(f)$ are equal.
\end{theorem}
In the correspondence, irreducible binary cubic forms correspond to orders in cubic fields.

The reader is referred to Chapter 2 of \cite{S72} for the following discussion of the space of binary cubic forms.  We have adopted the same conventions for ease of comparison.
$G_{\bR}$ acts naturally on the space
\begin{equation}
V_\bR = \left\{ax^3 + bx^2y + cxy^2 + dy^3: (a,b,c,d) \in
\bR^4\right\}
\end{equation}
of real binary cubic forms via, for $f \in V_{\bR}$ and $g \in G_{\bR}$,
\begin{equation}
g\cdot f(x,y) = f((x,y)\cdot g).
\end{equation}
Note that this differs by a factor of $\det g$ from the $G_{\bR}$ action on the basis $\langle 1, \omega, \theta\rangle$ of a cubic ring, the former action is called a twisted action.
The action $(x,y)\cdot g$ is a right action on $\bR^2$, so that the action on $V_{\bR}$ is a left action.
The discriminant
\begin{equation}
D = b^2c^2 + 18abcd- 4ac^3  - 4b^3d -27a^2d^2, 
\end{equation}
which is a homogeneous polynomial of degree four on $V_{\bR}$, is a relative
invariant:  $D(g\cdot f) = \chi(g)D(f)$
where $\chi(g) = \det(g)^6$. The dual space of $V_{\bR}$ is identified with
$V_{\bR}$ via the alternating
pairing \begin{equation}
\langle x, y \rangle = x_4y_1 - \frac{1}{3}x_3y_2 + \frac{1}{3}x_2y_3 -
x_1y_4.
\end{equation}
Let $\tau$ be the map $V_{\bR}\to V_{\bR}$ carrying each basis vector to its
dual basis
vector; the discriminant $\hat{D}$ on the dual space is normalized such that
$\tau$ is discriminant preserving.  There is an involution $\iota$ on $G_{\bR}$ given
by
\begin{equation}
g^\iota = \begin{pmatrix} 0&-1\\ 1 &0\end{pmatrix} (g^{-1})^t \begin{pmatrix}
0 & 1\\ -1 &0\end{pmatrix} = \frac{g}{\det g}.
\end{equation}
This satisfies, for all $g \in G_{\bR}$, $x \in V_\bR$, $y \in \hat{V}_\bR$,
\begin{equation}
\langle x, y \rangle = \langle g\cdot x, g^\iota \cdot y \rangle.
\end{equation}
Given $f \in L^1(V_{\bR})$ one has the Fourier transforms
\begin{align}
\hat{f}(x) &= \int_{V_\bR}f(y)e(-\langle x, y\rangle)dy.
\end{align}
If $\hat{f}$ also is in $L^1(V_{\bR})$ then
\begin{align}
f(x ) &= \frac{1}{9}\int_{V_\bR}\hat{f}(y)e(\langle x,y \rangle)dy.
\end{align}
Translation, dilation and the group action act on the Fourier transform as follows,
\begin{align} \label{translation_dilation_action}
 \int_{V_{\bR}} f\left(g \cdot \left(q^2 y + a\right)\right) e(-\langle x, y \rangle)dy& = \frac{e\left(\frac{\langle x, a\rangle}{q^2}\right)}{q^8} \int_{V_{\bR}}f(g \cdot y) e\left(-\left\langle \frac{x}{q^2}, y\right\rangle\right)dy\\
\notag & = \frac{e\left(\frac{\langle x, a \rangle}{q^2}\right)}{q^8 \chi(g)} \int_{V_{\bR}}f( y) e\left(-\left\langle \frac{ x}{q^2}, g^{-1}\cdot y\right\rangle\right)dy\\
\notag & = \frac{e\left(\frac{\langle x, a \rangle}{q^2}\right)}{q^8 \chi(g)} \int_{V_{\bR}}f( y) e\left(-\left\langle g^{\iota}\cdot \left(\frac{x}{q^2}\right), y\right\rangle\right)dy\\
\notag & = \frac{e\left(\frac{\langle x, a\rangle}{q^2}\right)\hat{f}\left(g^\iota \cdot \left(\frac{x}{q^2} \right) \right)}{q^8 \chi(g)}.
\end{align}

The set of forms of zero discriminant are called the singular set, $S$. The
non-singular forms split into spaces $V_+$ and $V_-$ of positive and negative
discriminant.  The space $V_+$ is a single $G^+$ orbit
with representative $x_+ = \left(\frac{1}{\sqrt{2}},0,\frac{-1}{\sqrt{2}},0\right)$, which has discriminant $1$ and stability group $I_{x_+}$ of order 3.
$V_-$ is also a single $G^+$ orbit with representative $x_- = \left(\frac{1}{\sqrt{2}},0,\frac{1}{\sqrt{2}},0\right)$ with discriminant $-1$ and
trivial stabilizer, see \cite{S72}, Proposition 2.2.  This orbit description is the reason that the shape of a cubic field may be identified with a group element in the real group action.

Set
\begin{equation}
w_1 = (0,0,1,0), \qquad w_2 = (0,0,0,1).
\end{equation}
The singular set is the disjoint union
\begin{equation}
S = \{0\}\sqcup G^1 \cdot w_1 \sqcup G^1 \cdot w_2.
\end{equation}
The stability group for the action of $G^1$ on $w_1$ is trivial $I_{w_1} =
\{1\}$, while on $w_2$ it is
\begin{equation}
I_{w_2}= N,
\end{equation}
see \cite{S72}, Proposition 2.3.

Over
$\zed$, the space of integral forms is a lattice $L$.  For each $m \neq 0$ the
set $L_m$ of integral forms of discriminant $m$ split into a finite
number $h(m)$ of $\SL_2(\zed)$ orbits. $h(m)$ is called the class number of
binary cubic forms of discriminant $m$.
With respect to the alternating pairing, the dual lattice $\hat{L}$ of $L$ is
given by forms
\begin{equation}
\hat{L} = \left\{ax^3 + bx^2 y + cxy^2 + dy^3: a,d \in \zed, b,c \in 3\zed
\right\}
\end{equation}
with middle coefficients divisible by 3. The class number of dual forms of
discriminant $m$, also finite, is indicated $\hat{h}(m)$.

Shintani obtained the following description of the singular integral forms.
\begin{lemma}\label{fibration_lemma} The singular forms  $\hat{L}_0$ are the disjoint union
\begin{align}
\hat{L}_0 &= \{0\}\sqcup\bigsqcup_{m=1}^\infty \bigsqcup_{\gamma \in
	\Gamma/\Gamma \cap N} \{\gamma \cdot (0,0,0,m)\}\sqcup \bigsqcup_{m=1}^\infty
\bigsqcup_{n=0}^{3m-1}\bigsqcup_{\gamma \in \Gamma}\{\gamma \cdot (0,0, 3m,n)\}.
\end{align}
Let
\begin{align}L_0(I) &= \bigsqcup_{m=1}^\infty \bigsqcup_{\Gamma/\Gamma\cap
	N}\{\gamma
\cdot (0,0,0,m)\},\\ \notag
\hat{L}_0(II) &= \bigsqcup_{m=1}^\infty
\bigsqcup_{n=0}^{3m-1}\bigsqcup_{\gamma \in \Gamma}\{\gamma \cdot (0,0,
3m,n)\}.\end{align}
\end{lemma}
\begin{proof}
 See \cite{S72}, Corollary to Proposition 2.10.
\end{proof}

The forms of positive discriminant break into two classes, the first of which
have stability group in $\Gamma$ which is trivial, and the second having
stability group of order 3.

For $m \in \zed_{\neq 0}$ choose $\{g_{i,m}\}_{1 \leq i \leq h(m)} \subset G^+,$
such that \begin{equation}\{x_{i,m} = g_{i,m}\cdot x_{\sgn \;m}\}_{1 \leq i \leq h(m)}\end{equation} are
representatives for
the $h(m)$ classes of binary cubic forms of discriminant $m$, similarly
$\{\hat{x}_{i,m} = \hat{g}_{i,m} \cdot x_{\sgn \; m}\}_{1 \leq i \leq \hat{h}(m)}$ a system of
representatives for the classes of dual forms.  The group elements $g_{i,m}$ are used in identifying the shape of the ring corresponding to $x_{i,m}$ with a point in $\Lambda_2 = \SL_2(\zed)\backslash \SL_2(\bR)$.  Set $\Gamma(i,m)< \Gamma$ the
stability group of $x_{i,m}$, similarly $\hat{\Gamma}(i,m)$.

For $f$ a Schwarz class function on $V_{\bR}$, the Fourier transform $\hat{f}$ is also Schwarz class, and the Poisson summation formula states that
\begin{equation}
 \sum_{x \in L}f(x) = \sum_{\xi \in \hat{L}} \hat{f}(\xi).
\end{equation}

\subsection{Estimates regarding the $q$-non-maximal set}

We say that a cubic ring is maximal if it is not a proper subring of another cubic ring.  This is a property which may be checked locally.  We say that a cubic ring $R$ is maximal at $p$ if $R \otimes \zed_p$ is maximal as a cubic ring over $\zed_p$, a condition which is determined by congruences modulo $p^2$.  Let $\Phi_p$ be the indicator function of $a \in V_{p^2} = (\zed/p^2\zed)^4$ such that the binary cubic form $a$ corresponds to a non-maximal cubic ring.  Note that this function is constant on $\SL_2(\zed)$ orbits. Define the Fourier transform of $\Phi_p$ by, for $b \in (\zed/p^2\zed)^4$,
\begin{equation}
\hat{\Phi}_p(b) = \frac{1}{p^8} \sum_{a \in (\zed/p^2\zed)^4} \Phi_p(a) e\left(\frac{\langle b, a\rangle}{p^2}\right).
\end{equation}
The Fourier transform of $\Phi_p$ has been explicitly determined depending on the factorization type of the discriminant $P$ at $b$ in $\zed/p^2\zed$ in \cite{TT13a}.  To state this, the orbits of $V_{p^2}$ are classified, with representatives $b$, chosen up to multiplication by a multiplicative unit. Note that entries can describe multiple orbits.

\begin{tabular}{|l|l|l|}
 \hline
 type of $b$ & $b$ & condition on the coefficients\\
 \hline
 $(1_{**}^3)$ & $(1,0,0,0)$ & -\\
 \hline
 $(1_*^3)$ & $(1, 0, \ell, 0)$& $\ell \in p R^\times$\\
 \hline
 $(1_{\max}^3)$ & $(1, 0, k, -\ell)$ & $k \in pR, \ell \in p R^\times$\\
 \hline
\end{tabular}

\begin{lemma}
 Let $p>3$. Let $b \in pV_{p^2}$.  Write $b = pb'$, and regard $b'$ as an element of $V_p = \left(\zed/p\zed \right)^4$. Then
 \begin{equation}
  \hat{\Phi}_p(pb') = \left\{\begin{array}{lll}p^{-2} + p^{-3}-p^{-5}, && b' \text{ of type } (0)  \\ p^{-3}-p^{-5}, && b' \text{ of type } (1^3), (1^2 1) \\ -p^{-5}, && b' \text{ of type } (111), (21), (3)\end{array}\right.
 \end{equation}
 where the type refers to the factorization type of $b'$.
 
 For $b \in V_{p^2}\setminus p V_{p^2}$, 
 \begin{equation}
  \hat{\Phi}_p(b) = \left\{\begin{array}{lll} p^{-3}-p^{-5}, && b \text{ of type } (1^3_{**})\\ -p^{-5}, && b \text{ of type } (1_*^3), (1_{\max}^3)\\ 0, && \text{otherwise.}\end{array}\right.
 \end{equation}

\end{lemma}

The above evaluations are used in the following way.

\begin{lemma}\label{FT_eval_lemma}
Let $p > 3$.  For $m, n \in \zed$, we have the evaluations
 \begin{align}
 \hat{\Phi}_p(0,0,0,m) &= \left\{\begin{array}{lll} p^{-2} + p^{-3} -p^{-5}, && p^2|m\\ p^{-3} -p^{-5}, && p^2\nmid m \end{array}\right.\\
 \notag \hat{\Phi}_p(0,0,m,n) &= \left\{\begin{array}{lll} p^{-2} + p^{-3} -p^{-5}, && p^2|(m,n)\\ p^{-3} -p^{-5}, && p^2\nmid (m,n), p|m\\ 0, && p\nmid m \end{array}\right..
\end{align}
\end{lemma}
\begin{proof}
If $p^2$ divides the form $b$, then $\hat{\Phi}_p(b) = p^{-2} + p^{-3} - p^{-5}$.  In the case of $(0,0,0,m)$, if $p|m$ but $p^2\nmid m$ then $b$ is of type $pb'$ with $b'$ of type $(1^3)$ modulo $p$.  If $p\nmid m$ then $b$ is of form $(1_{**}^3)$.  In either case, $\hat{\Phi}_p(0,0,0,m) = p^{-3}-p^{-5}$.  In the case of $(0,0,m,n)$, if $p \nmid m$ then $b\not \in p V_{p^2}$ and $b$ is not of type $(1^3)$ modulo $p$, so $b$ is not in the orbit of $(1_{**}^3), (1_*^3), (1_{\max}^3)$, so $\hat{\Phi}_p(b) = 0$. In the second case, if $p|m$ but $p \nmid n$ then $b$ is $\GL_2(\zed/p^2\zed)$ equivalent to a form in $(1^3_{**})$.  If $p$ divides $b$ but $p^2$ does not, then $b= pb'$ with $b'$ equivalent to a form of type either $(1^3)$, $(1^21)$.
\end{proof}

Extend $\Phi_p$ to $\Phi_q$ multiplicatively when $q = p_1p_2...p_n$ is square-free, 
\begin{equation}
\Phi_q(x) = \prod_{i=1}^n \Phi_{p_i}(x). 
\end{equation}
It follows from Proposition 4.11 of \cite{TT13a} that for square-free $q$, the Fourier transform also factors as a product,
\begin{equation}
 \hat{\Phi}_q(b) = \prod_{p|q} \hat{\Phi}_p(b).
\end{equation}

We quote the following consequence from \cite{TT13b}.

\begin{theorem}\label{FT_bound_theorem}
	Uniformly in $q$ and $Y$, for all $\epsilon > 0$,
	\begin{equation}
	\sum_{0 < |m| \leq Y}\sum_{i=1}^{\hat{h}(m)} \frac{\left|\hat{\Phi}_q(\hat{x}_{i,m})\right|}{\left|\hat{\Gamma}(i,m)\right|} \ll_\epsilon q^{-7+\epsilon}Y.
	\end{equation}
\end{theorem}
\begin{proof}
 This is stated as a consequence of Theorem 3.1 of \cite{TT13b} in eqn. (3.4), without including the factor of the inverse of the stabilizer, which can only decrease the sum.
\end{proof}

\section{Background regarding automorphic forms on $\GL_2$}
This section reviews the theory of automorphic forms based on the discussion in \cite{LRS08}.  We include a discussion of the right $K = \SO_2(\bR)$-type which gains a further parameter of equidistribution in Theorem \ref{cubic_field_theorem} compared to the earlier works \cite{T97} and \cite{BH16}. Here and in what follows, a form $\phi$ of right $K$-type $2k$ satisfies $\phi(gk_{\theta}) = e^{i 2k \theta} \phi(g)$, and the map $\theta \mapsto e^{i2k\theta}$ is called a character of degree $2k$ for $K$.  What we need from this theory is that the fact that a cusp form $\phi$ may be represented as a function on the modular surface $\Gamma \backslash \bH$ satisfying an exponential decay condition in the cusp, that it has a Fourier expansion in the parabolic direction, together with some estimates for the functions and coefficients appearing in the Fourier expansion.

  The space of mean-zero functions
\begin{equation}
L_0^2(\Gamma \backslash G^1) = \left\{f \in L^2(\Gamma \backslash G^1): \int_{\Gamma \backslash G^1}f(g) dg = 0 \right\}
\end{equation}
splits as
\begin{equation}
L_0^2(\Gamma \backslash G^1) = L^2_{\cusp}(\Gamma \backslash G^1) \oplus L^2_{\Eis}(\Gamma \backslash G^1)
\end{equation}
where 
\begin{equation}\label{def_cusp_form}
L_{\cusp}^2(\Gamma \backslash G^1) = \left\{f \in L^2_0(\Gamma \backslash G^1): \int_{(N' \cap \Gamma) \backslash N'}f(\nu x)d\nu = 0, a.e. \;x \in \Gamma \backslash G^1 \right\}
\end{equation}
and where $L_{\Eis}^2(\Gamma \backslash G^1)$ is spanned by the incomplete Eisenstein series.

\subsection{The cuspidal spectrum}
We follow the discussion of \cite{LRS08} Section 3.1. Let $\pi$ be an infinite dimensional unitary irreducible representation of $\SL_2(\bR)$ in a Hilbert space $\sH$, which factors through $\PSL_2(\bR)$, and let $\sH^{(K)}$ be the $K$-finite vectors in $\sH$.  Then $\sH^{(K)}$ is spanned by  orthogonal one dimensional subspaces $\sH_n$, $n$ even,  transforming under $K$ on the right by the character of degree $n$, the $K$-type.  These $\sH_n$ consist of smooth vectors and may be chosen to be joint eigenfunctions of the Casimir operator and of the Hecke algebra.

Let $\sH_n = \bC \phi_n$.  For some parameter $s$, one has the action of the Lie algebra given by
\begin{align}
d \pi(W)\phi_n &= in \phi_n\\\notag
d\pi (E^-)\phi_n &= (s + 1-n)\phi_{n-2}\\\notag
d\pi(E^+)\phi_n &= (s+1+n) \phi_{n+2}
\end{align}
where
\begin{equation}
H = \begin{pmatrix} 1&0\\0 &-1\end{pmatrix}, \qquad V = \begin{pmatrix}0 &1\\1 &0\end{pmatrix}, \qquad W = \begin{pmatrix} 0 &1\\ -1 &0\end{pmatrix}
\end{equation}
and $E^{\pm} = H \pm iV$. The operators $E^+$ and $E^-$ are called raising and lowering operators.

The classification now breaks into two cases.
\begin{enumerate}
	\item (Maass case) There is a $K$-invariant vector $\phi_0$, which is called the minimal vector.  In this case, $s \in i\bR$ and there is no highest or lowest $K$-type, so
	\begin{equation}
	\sH^{(K)} = \bigoplus_{n \in \zed} \sH_{2n}.
	\end{equation}
	\item (Holomorphic case) When $\sH$ has a lowest $K$-type $m_0$ one has $m_0 > 0$ even.  In this case, $s = m_0-1$ and 
	\begin{equation}
	d\pi(E^-)\phi_{m_0} = 0.
	\end{equation}
	One has
	\begin{equation}
	\sH^{(K)} = \bigoplus_{m = m_0, \text{even}}^\infty \sH_{m}.
	\end{equation}
	$\phi_{m_0}$ is called the lowest weight vector, or minimal vector.  
	
	When $\sH$ has a highest $K$-type $m_0$ one has $m_0 < 0$ even.  In this case, $s = -m_0 -1$ and
	\begin{equation}
	d\pi(E^+)\phi_{m_0} = 0.
	\end{equation}
	Now
	\begin{equation}
	\sH^{(K)} = \bigoplus_{m=-\infty, \text{even}}^{m_0} \sH_m
	\end{equation}
	$\phi_{m_0}$ is called the highest weight or maximal vector.
\end{enumerate}
We extend the automorphic form $\phi$ to a function on $G^+$ by requiring that $\phi$ be invariant under scaling. Since all of the forms which we work with have even $K$-type $2k$, it follows that 
\begin{equation}
 \phi(g) = \phi\left(\frac{g}{\det g} \right) = \phi\left(\begin{pmatrix} 0 & 1\\ -1 &0\end{pmatrix} (g^{-1})^t \begin{pmatrix} 0 &-1\\ 1 &0\end{pmatrix} \right) = (-1)^{k} \phi\left( (g^{-1})^t \right)
\end{equation}
and thus $\phi\left(g^{-1}\right) =(-1)^k \phi\left(g^t\right)$, see the discussion in \cite{LRS08} following Proposition 3.1.  
\subsubsection{Upper half plane model}
For this section, see \cite{LRS08} Section 4.

The representation spaces of $K$-type $k$ can be realized as automorphic functions of weight $k$ on the upper half plane $\bH$, $f: \bH \to \bC$, which satisfy the automorphy relation under fractional linear transformations given by 
\begin{equation}
f(\gamma z) = \left(\frac{cz+d}{|cz+d|}\right)^k f(z), \qquad \gamma = \begin{pmatrix} a&b\\ c&d\end{pmatrix} \in \Gamma
\end{equation}
with an exponential decay condition in the cusp.  To realize the weight $k$ holomorphic cusp forms in this model, multiply by $y^{\frac{k}{2}}$.

The Casimir operator is realized in this model as the weight $k$ Laplacian
\begin{equation}
\Delta_k = y^2 \left(\frac{\partial^2}{\partial x^2} + \frac{\partial^2}{\partial y^2}\right) - iky \frac{\partial}{\partial x}.
\end{equation}
Denote $W_{\kappa, \mu}$ the Whittaker function, which satisfies
\begin{equation}
W_{\kappa, \mu}(y) \sim y^\kappa e^{-\frac{y}{2}}, \qquad y \to \infty
\end{equation}
and as $y \downarrow 0$,
\begin{equation}
W_{\kappa, \mu}(y) \sim \frac{\Gamma(-2\mu)}{\Gamma(\frac{1}{2}-\mu-\kappa)}y^{\frac{1}{2} + \mu} + \frac{\Gamma(2\mu)}{\Gamma(\frac{1}{2}+\mu-\kappa)} y^{\frac{1}{2}-\mu}, \qquad \mu \neq 0.
\end{equation}
When $\mu = 0$, $W_{\kappa, 0}(y) \ll y^{\frac{1}{2}}\log y$ as $y \downarrow 0$.
For later reference, abbreviate  $f_k^{\pm}(z,s) = W_{\pm \frac{k}{2}, s-\frac{1}{2}}(4\pi y)e(\pm x)$, which is an eigenfunction of $\Delta_k$ with eigenvalue $\lambda = s(1-s)$. Note that the Whittaker function may be recovered from $f_k^{\pm}$ at imaginary argument,
\begin{equation}
 f_k^{\pm}(iy,s)= W_{\pm \frac{k}{2}, s-\frac{1}{2}}(4\pi y).
\end{equation}
This is used to obtain linear relations among the derivatives of the Whittaker functions via  raising and lowering operators below.

A weight $k$ Hecke-eigen-cusp-form $\phi$ of eigenvalue $\lambda = s(1-s)$ has a Fourier expansion
\begin{align}
\phi(z) &= \frac{1}{2}\sum_{n\neq 0} \frac{\rho_\phi(n)}{\sqrt{|n|}} W_{\frac{\sgn(n)k}{2}, s-\frac{1}{2}}(4\pi|n|y)e(nx) \\&\notag= \frac{1}{2} \sum_{n \neq 0} \frac{\rho_\phi(n)}{\sqrt{|n|}}f_k^{\sgn n}(z,s).
\end{align}
Note that the Fourier coefficients $\rho_\phi(n)$ differ from those in \cite{LRS08} by a factor of $|n|^{\frac{1}{2}}$.
In the case $k = 0$, $W$ is expressed in terms of the $K$-Bessel function by
\begin{equation}
 \frac{1}{2}W_{0, it}(4\pi x) = \sqrt{x} K_{it}(2\pi x).
\end{equation}
For consistency with the weight zero case, we define
\begin{equation}
 \frac{1}{2}W_{\kappa, it}(4\pi x) = \sqrt{x}\check{K}_{\kappa, it}(2\pi x)
\end{equation}
so that the Fourier development may be written in general
\begin{equation}
 \phi(z) = y^{\frac{1}{2}}\sum_{n\neq 0} \rho_\phi(n) \check{K}_{\frac{\sgn(n)k}{2}, s-\frac{1}{2}}(2\pi|n|y)e(nx).
\end{equation}

In the upper half plane model the raising and lowering operators are given by
\begin{align}
K_k &= \frac{k}{2} + y \left(i \frac{\partial}{\partial x} + \frac{\partial}{\partial y}\right), \qquad 
\Lambda_k = \frac{k}{2} + y\left(i \frac{\partial}{\partial x} - \frac{\partial}{\partial y}\right).
\end{align}
$K_k$ takes forms of weight $k$ to weight $k+2$, $\Lambda_k$ takes forms of weight $k$ to weight $k-2$.  These satisfy
\begin{equation}
K_k \Delta_k = \Delta_{k+2}K_k, \qquad \Lambda_k \Delta_k = \Delta_{k-2}\Lambda_k
\end{equation}
and operate on the Fourier expansion via
\begin{align}
&K_kf_k^+(z,s)=-f_{k+2}^+(z,s), \qquad K_k f_k^-(z,s) = \left(s + \frac{k}{2}\right)\left(1-s + \frac{k}{2}\right)f_{k+2}^-(z,s),\\ \notag
&\Lambda_kf_k^+(z,s) = -\left(s-\frac{k}{2}\right)\left(1-s -\frac{k}{2}\right) f_{k-2}^+(z,s), \qquad \Lambda_k f_k^-(z,s) = f_{k-2}^-(z,s).
\end{align}

The feature of $\check{K}$ which we will need is as follows.
\begin{lemma}\label{Mellin_transform_lemma}
In $\RE(s)>0$ the Mellin transform
\begin{equation}
 \widetilde{K}_{\kappa, it}(s) = \int_0^\infty \check{K}_{\kappa, it}(x) x^{s-1}dx
\end{equation}
is holomorphic.  For each fixed $\sigma >0$ and $A > 0$,
\begin{equation}
 |\tau|^A\left|\widetilde{K}_{\kappa, it}(\sigma + i\tau)\right| \to 0
\end{equation}
as $|\tau| \to \infty$.
\end{lemma}

\begin{proof}
 The holomorphicity follows from the fact that the Whittaker functions are smooth and from the asymptotic behavior at 0 and $\infty$.  When $k = \kappa = 0$, the $K$-Bessel function has Mellin transform
 \begin{equation}
  \int_0^\infty K_{it}(x)x^{\sigma + i\tau-1} dx = 2^{\sigma + i\tau - 2} \Gamma\left(\frac{\sigma + i\tau + it}{2} \right) \Gamma\left(\frac{\sigma + i\tau -it}{2} \right),
 \end{equation}
which satisfies the decay condition due to the decay of the Gamma function in vertical strips.  For other even $k$, applying the raising or lowering operators $\frac{k}{2}$ times to the functions $f_*^{\pm}(iy, s)$ expresses $\check{K}_{\kappa, it}(x)$ as a linear combination of terms of the type $x^a K^{(b)}_{it}(x)$, $a \geq b$ involving powers of $x$ and derivatives of the $K$-Bessel function.  This expresses $\widetilde{K}_{\kappa, it}(s)$ as a bounded linear combination of Mellin transforms of the $K$-Bessel function, from which the claimed decay follows.
 
\end{proof}

Assume that $\phi$ is Hecke-normalized, so that $\rho_\phi(1) = 1$. 
Then the Fourier coefficients satisfy the Hecke multiplicativity relation
\begin{equation}
 \rho_\phi(m)\rho_\phi(n) = \sum_{d |(m,n)}\rho_\phi\left(\frac{mn}{d^2} \right).
\end{equation}
$\rho_\phi(n)$ is the eigenvalue of the $n$th Hecke operator $T_n$. Note that the Casimir eigenvalue, Hecke eigenvalues and weight $k$ are sufficient to recover the form $\phi$ via the Fourier expansion (multiplicity 1).
Attached to the form $\phi$ is the $L$-function $L(s, \phi)$,
\begin{equation}
 L(s,\phi) = \sum_{n=1}^\infty \frac{\rho_\phi(n)}{n^s}, \qquad \RE(s)>1.
\end{equation}
This extends to an entire function and is given by an absolutely convergent Euler product in $\RE(s)>1$.

\subsection{Spherical kernels} 
Say that a function $f$ on $\SL_2(\bR)$ transforms on the left, resp. right by a character of $\SO_2(\bR)$ of degree $2k$ if, for all $g \in \SL_2(\bR)$, $f(k_\theta g) = e^{i2k\theta} f(g)$, resp. $f(g k_\theta) = e^{i2k\theta}f(g)$.  A function which transforms on the left by a character of degree $2k_1$ and on the right by a character of degree $2k_2$ is said to be spherical of $K$-type $(2k_1, 2k_2)$.  

\begin{lemma}
 Let $\phi$ be a Hecke-eigen cusp form transforming on the right by a character of $\SO_2(\bR)$ of degree $2k$.  Let $f$ be defined on $\SL_2(\bR)$, smooth, and of compact support, and spherical of $K$-type $(2k,2k)$.  There is a constant $\Lambda_{f, \phi}$ depending on $f$ and $\phi$ such that, for all $g_0 \in \SL_2(\bR)$
 \begin{equation}
  \int_{\SL_2(\bR)} f(gg_0)\phi\left(g^{-1}\right) dg = \Lambda_{f,\phi}\phi(g_0).
 \end{equation}
Moreover, for an appropriate choice of $f$, $\Lambda_{f,\phi} \neq 0$.
\end{lemma}

\begin{proof}
 Define $F\phi(g_0) = \int_{\SL_2(\bR)} f(gg_0)\phi(g^{-1}) dg.$  Note that the right $K$-type of $F\phi$ is $2k$ by passing the transformation under the integral and applying the right $K$-type of $f$.  After a change of variable,
 \begin{equation}
  F\phi(g_0) = \int_{\SL_2(\bR)} f(h)\phi(g_0 h^{-1}) dh.
 \end{equation}
This representation shows that $F\phi$ is left invariant under $\SL_2(\zed)$.
By applying the Hecke operators and Casimir operator to $\phi$ under the integral, it follows that $F\phi$ has the same Hecke and Casimir eigenvalues as $\phi$, and hence, by strong multiplicity 1, that $F\phi = \Lambda_{f,\phi} \phi$ for some constant $\Lambda_{f,\phi}$, see \cite{G06}.  To demonstrate that the constant can be taken non-zero, write $h = k_{\theta_1}a_tk_{\theta_2} = k_{\theta_1}k_{\theta_2}(a_t)^{k_{\theta_2}} $, where $(a_t)^{k_{\theta_2}} = k_{-\theta_2}a_t k_{\theta_2}$.  When $t = 1$, integrating in $\theta_1, \theta_2$ evaluates to 1, since $(a_1)^{k_{\theta_2}} = a_1$ and $\phi$ transforms on the right by the same character of $\SO_2(\bR)$ as $f$ does on the left.  Hence if the support of $f$ is sufficiently close to $\SO_2(\bR)$, and $f$ is non-negative on $A$, then it can be arranged that $F\phi \neq 0$. 

\end{proof}

\section{The orbital integral representation}
We now introduce the main analytic objects of study, which are generating functions for cubic rings which are non-maximal at all primes $p|q$, and which are twisted by an automorphic cusp form $\phi$.  The orbital integral representation given here was given in \cite{H17} and is a modification of the construction of \cite{S72}.

Let $f$ be a function supported on either $V_+$ or $V_-$.  We assume that $f$ takes the form
\begin{equation}
 f(x) = f_D(|\Disc(x)|) \sum_{g \in G^+: g \cdot x_{\pm}= x} f_G\left(g \right).
\end{equation}
where $f_G$ is spherical of $K$-type $(2k,2k)$ on $\SL_2(\bR)$, extended to $\GL_2(\bR)$ as invariant under multiplication by a scalar, and $f_D$ is a function on $\bR^\times$.

Assume that $f$ is smooth and compactly supported.  As a result, the Fourier transform $\hat{f}$ is Schwarz class.  
Let $\phi$ be an automorphic form on $\SL_2(\zed)\backslash \SL_2(\bR)$, and for square-free $q$,   define the $q$-non-maximal orbital $L$-function by
\begin{equation}
 Z_q^{\pm}(f, \phi, L; s) = \int_{G^+ / \Gamma} \chi(g)^s \phi\left(g^{-1}\right) \sum_{x \in L} \Phi_q(x)f(g\cdot x) dg,
\end{equation}
where $\chi(g) = (\det g)^6$.

Introduce the twisted Shintani $\sL$-functions, defined for $\RE(s) > 1$ by
\begin{align}
 \sL_q^+(\phi, s) &= \sum_{m=1}^\infty \frac{1}{m^s} \sum_{g_0 \in I_{x_+}}\sum_{i=1}^{h(m)}\Phi_q(x_{i,m})\frac{\phi(g_{i,m}g_0)}{|\Gamma(i,m)|}, \\ \notag
 \sL_q^-(\phi, s) &= \sum_{m=1}^\infty \frac{1}{m^s} \sum_{i=1}^{h(-m)} \Phi_q(x_{i,-m}) \phi(g_{i,-m}).
\end{align}
These functions are well-defined, since $\Phi_q$ is $\Gamma$-invariant; the series are absolutely convergent in $\RE(s)>1$ since $\phi$ and $\Phi_q$ are bounded and the original Shintani zeta functions, which are obtained by omitting these factors, converge absolutely in $\RE(s)>1$, \cite{S72}.

The twisted $\sL$-functions may be recovered from the orbital representation as follows. 
\begin{lemma}
In $\RE(s) >1$, 
 \begin{align}
  Z^{\pm}_q(f, \phi, L; s) &=  \frac{\Lambda_{f_G, \phi}}{12} \sL_q^{\pm}(\phi, s) \tilde{f}_D(s).
 \end{align}

\end{lemma}

\begin{proof}
 Write
 \begin{align}
  &Z^+_q(f, \phi, L;s) = \int_{g \in G^+/\Gamma} \chi(g)^s \phi\left(g^{-1}\right) \sum_{x \in L} \Phi_q(x) f(g \cdot x) dg\\ \notag
  &=\int_{G^+/\Gamma} \chi(g)^s \phi\left(g^{-1}\right) \sum_{m=1}^\infty \sum_{i=1}^{h(m)}\frac{\Phi_q(x_{i,m})}{|\Gamma(i,m)|} \sum_{\gamma \in \Gamma} f(g \gamma  g_{i,m} \cdot x_+)dg.
  \end{align}
  Unfold the sum over $\Gamma$ and the integral to obtain
  \begin{align}    
  &Z^+_q(f, \phi, L;s)= \int_{G^+}\sum_{m=1}^\infty \sum_{i=1}^{h(m)} \frac{\Phi_q(x_{i,m})}{|\Gamma(i,m)|}  \chi(g)^s\phi\left(g^{-1}\right)\sum_{g_0\in I_{x_+}} f_G(g g_{i,m}g_0) f_D(\chi(g) m) dg\\ 
  \notag &= \sum_{m=1}^\infty \sum_{i=1}^{h(m)} \frac{\Phi_q(x_{i,m})}{|\Gamma(i,m)|} \int_{G^1}\phi(g^{-1}) \sum_{g_0 \in I_{x_+}} f_G(g  g_{i,m}  g_0) dg \int_{\bR^\times} \lambda^{12 s} f_D(\lambda^{12}m) \frac{d\lambda}{\lambda}.
  \end{align}
  The exchange of order of summation and integration is justified by the compact support of the test function, which makes the integral over $G^1$ bounded in $L^1$, so that the method of Shintani \cite{S72} applies.  Since $f_G$ is conjugation invariant, integration in $g$ may be interpreted as left convolution
  \begin{align}
  \int_{G^1} f_G(gg_{i,m}g_0) \phi\left(g^{-1}\right) dg &= \int_{G^1}f_G(g_{i,m}g_0 g) \phi\left(g^{-1}\right)dg \\\notag&= \int_{G^1}f_G\left(g_{i,m}g_0 g^{-1} \right)\phi(g)dg
  \\\notag&= \Lambda_{f_G,\phi} \phi(g_{i,m}g_0).
  \end{align}
  Evaluating both integrals obtains
  \begin{align}
    Z^+_q(f, \phi, L;s)&= \frac{\Lambda_{f_G, \phi}}{12} \sL^+_q(\phi, s) \tilde{f}_D(s)
 \end{align}
which gives the claimed formula.  The proof in the case of $Z^-$ is similar.
\end{proof}

Introduce the truncated orbital functions
\begin{align}
 Z^{\pm, +}_q(f, \phi, L; s) &= \int_{\substack{G^+/\Gamma\\ \chi(g) \geq 1}} \chi(g)^s \phi\left(g^{-1}\right) \sum_{x \in L}\Phi_q(x) f(g\cdot x) dg\\ \notag
 \hat{Z}^{\pm, +}_q(\hat{f}, \phi, \hat{L};1-s) &= \int_{\substack{G^+/\Gamma\\ \notag \chi(g) \geq 1}} \chi(g)^{1-s} \phi\left(g^{-1}\right) \sum_{x \in \hat{L}\setminus \hat{L}_0} \hat{\Phi}_q(x) \hat{f}\left(g \cdot \frac{x}{q^2} \right)dg
\end{align}
and the singular part
\begin{equation}
Z^{\pm, 0}_q(\hat{f}, \phi, \hat{L}; s) =\int_{\substack{G^+/\Gamma\\ \chi(g) \leq 1}} \chi(g)^{s-1} \phi\left(g^{-1}\right) \sum_{x \in \hat{L}_0} \hat{\Phi}_q(x) \hat{f}\left(g^\iota \cdot\frac{x}{q^2}\right) dg.
\end{equation}

Note that $Z_q^{\pm, +}$ and $\hat{Z}_q^{\pm, +}$ are entire due to the compact support of $f$ and the rapid decay of $\hat{f}$.

\begin{lemma}[Split functional equation]
 In $\RE(s)>1$,
 \begin{align}
  Z_q^{\pm}(f, \phi, L;s) &= Z_q^{\pm, +}(f, \phi, L; s) + \hat{Z}_q^{\pm, +}(\hat{f}, \phi, \hat{L}; 1-s)+ Z^{\pm, 0}_q(\hat{f}, \phi, \hat{L}; s).
 \end{align}

\end{lemma}

\begin{proof}
  In the orbital integral representation
  \begin{equation}
   Z_q^{\pm}(f, \phi, L; s) = \int_{G^+ / \Gamma} \chi(g)^s \phi\left(g^{-1}\right) \sum_{x \in L} \Phi_q(x)f(g\cdot x) dg,
  \end{equation}
 split the integral at $\chi(g) = 1$.  The part of the integral with $\chi(g) \geq 1$ is $ Z^{\pm, +}_q(f, \phi, L; s)$.  Since the part of the integral with $\chi(g)=1$ has measure 0, write the second part of the integral as
 \begin{equation}
  \int_{G^+ / \Gamma, \chi(g) \leq 1} \chi(g)^s \phi\left(g^{-1}\right) \sum_{x \in L} \Phi_q(x)f(g\cdot x) dg.
 \end{equation}
 The Poisson summation formula permits the representation (see (\ref{translation_dilation_action}))
\begin{align}
 \sum_{x \in L} \Phi_q(x)f(g\cdot x) &= \sum_{a \in V(\zed/q^2 \zed)} \Phi_q(a) \sum_{x \in L} f(g \cdot(q^2 x +a))\\ \notag
 &= \frac{1}{\chi(g)q^8} \sum_{a \in V(\zed/q^2 \zed)} \Phi_q(a) \sum_{y \in \hat{L}} e \left(\left\langle \frac{y}{q^2}, a \right\rangle\right) \hat{f}\left(g^{\iota} \cdot\frac{y}{q^2}\right)\\ \notag
 &= \frac{1}{\chi(g)}\sum_{y \in \hat{L}} \hat{\Phi}_q(y) \hat{f}\left(g^\iota \cdot\frac{y}{q^2}\right).
\end{align}
The part of the sum from $\hat{L}_0$ contributes $Z^{\pm, 0}_q(\hat{f}, \phi, \hat{L}; s).$  In the remainder of the sum, make the change of variable $g^\iota := g$, to obtain $\hat{Z}^{\pm, +}_q(\hat{f}, \phi, \hat{L};1-s)$.
\end{proof}

\section{Treatment of the singular integral}
We now check that the split functional equation gives the holomorphic continuation of the singular part of the orbital integral to all of $\bC$ and study the $q$-dependence. This section closely follows \cite{H17}, but note that by enforcing that $f|_S = 0$ we only consider the contribution from $\hat{f}$.  Continue to assume that $\hat{f}$ is Schwarz class.  We  assume that $\phi$ is a Hecke-eigen cusp form of right-$K$-type $2k$.

For $g \in G^1/\Gamma$ define (note $g = g^\iota$ since $g \in G^1$)
\begin{equation}
 J_q\left(\hat{f}\right)(g) = \sum_{x \in \hat{L}_0}\hat{\Phi}_q(x) \hat{f}\left(\frac{g\cdot x}{q^2}\right).
\end{equation}

 Following \cite{S72} write $g \in G^1$ in the Iwasawa decomposition as $g = k_\theta a_t n_u$ and define $t(g) = t$.  Let $\fS_C$ denote the Siegel set
 \begin{equation}
  \fS_C = \left\{k_\theta a_t n_u: \theta \in \bR, t \geq C, |u| \leq \frac{1}{2}\right\},
 \end{equation}
and define the class of functions
\begin{equation}
 C(G^1/\Gamma, r) = \left\{f \in C(G^1/\Gamma): \sup_{g \in \fS_{\frac{1}{2}}} t(g)^r |f(g)| < \infty\right\}.
\end{equation}
\begin{lemma}
 Let $f = f_{G} \otimes f_D$ have compact support, and suppose that $\hat{f}(x) \ll \frac{1}{1 + \|x\|_2^A}$ for some $A > 4$.  Then
 \begin{equation}
  J_q\left(\hat{f}\right) \in C(G^1/\Gamma, A-6).
 \end{equation}

\end{lemma}
\begin{proof}
 As in Lemma 2.10 of \cite{S72}, write $g \in \fS_{\frac{1}{2}}$ as $k_\theta a_t n_u = c(g) a_t$ where $c(g) = k_\theta a_t n_u a_{\frac{1}{t}}$ and note that, as $g$ varies in $\fS_{\frac{1}{2}}$, $c(g)$ varies in a compact set.  Hence, by compactness, 
 \begin{equation}
 \inf_{\substack{g \in \fS_{\frac{1}{2}},\\ 0 \neq x \in V_{\bR}}}\left\{ \frac{\|c(g) \cdot x\|_2}{\|x\|_2}\right\} > 0.
 \end{equation}
 
 By Poisson summation
 \begin{equation} 
  J_q\left(\hat{f}\right)(g) = \sum_{x \in L \setminus L_0}\Phi_q(x) f(g \cdot x) - \sum_{x \in \hat{L} \setminus \hat{L}_0} \hat{\Phi}_q(x)\hat{f}\left(\frac{g\cdot x}{q^2}\right).
 \end{equation}
Since we restrict to $x$ outside the singular set, for each such $x = (x_1, x_2, x_3, x_4)$, at least one of $x_1, x_2$ is non-zero.  Since $a_t \cdot x = (t^3 x_1, tx_2, t^{-1}x_3, t^{-3}x_4)$, the sum over $L \setminus L_0$ vanishes if $t$ is sufficiently large, by the compact support. Meanwhile, in the dual sum, \begin{equation}\left\|a_t\cdot \frac{x}{q^2}\right\|_2 \geq \frac{\min(t^3, 3t)}{q^2},\end{equation} and hence the dual sum may be bounded by splitting on $x_1  \neq 0$ and $x_1 = 0$,
\begin{align}
 \ll_q \sum_{\substack{x_1, x_2, x_3, x_4 \\ (x_1, x_2) \neq (0,0)}} \frac{1}{1 + \|(t^3 x_1, t x_2, t^{-1}x_3, t^{-3}x_4)\|_2^A} &\ll_q \frac{1}{t^{3A}} \sum_{x_1 \neq 0} \frac{1}{1 + \|(x_1, t^{-2}x_2, t^{-4}x_3, t^{-6}x_4)\|_2^A}\\\notag& \qquad+ \frac{1}{t^A} \sum_{x_2 \neq 0} \frac{1}{1 + \|(0,x_2, t^{-2}x_3, t^{-4}x_4)\|_2^A} \\\notag&\ll_q \frac{1}{t^{A-6}}.
\end{align}

\end{proof}
The object of interest is
\begin{equation}\label{singular_integral}
 \sI\left(\hat{f}, \phi\right) = \int_{G^1/\Gamma} \phi(g^{-1})J_q\left(\hat{f}\right)(g) dg.
\end{equation}
Note that there is not a question of convergence when $\phi$ is a cusp form due to the exponential decay in the cusp.
Recall the decomposition of Lemma \ref{fibration_lemma},
\begin{align}
 \hat{L}_0 &= \{0\} \sqcup \bigsqcup_{m=1}^\infty \bigsqcup_{\gamma \in \Gamma/(\Gamma \cap N)} \{\gamma \cdot (0,0,0,m)\} \sqcup \bigsqcup_{m=1}^\infty \bigsqcup_{n=0}^{3m-1} \bigsqcup_{\gamma \in \Gamma} \{\gamma \cdot (0,0,3m,n)\} \\\notag&= \{0\} \sqcup L_0(I) \sqcup \hat{L}_0(II).
\end{align}
The contribution from $\{0\}$ to (\ref{singular_integral}) is 0, since $\phi$ is orthogonal to the constant function.
Let 
\begin{align}
\Theta_{q}^{(1)}(\phi) &= \int_{G^1/\Gamma} \phi\left(g^{-1}\right) \sum_{x \in L_0(I)}\hat{\Phi}_q(x)\hat{f}\left(g \cdot \frac{x}{q^2}\right)dg\\
\notag \hat{\Theta}_{ q}^{(2)}( \phi) &= \int_{G^1/\Gamma}  \phi\left(g^{-1}\right) \sum_{x \in \hat{L}_0(II)}\hat{\Phi}_q(x) \hat{f}\left(g\cdot \frac{x}{q^2}\right)dg.
\end{align}

\begin{lemma}
 Let $\phi$ be a cusp form, of right $K$-type $2k$, which is an eigenfunction of the Hecke algebra.  We have
 \begin{equation}
  \Theta_{q}^{(1)}( \phi) = 0.
 \end{equation}

\end{lemma}

\begin{proof}
	Unfold the integral and sum to write, using $\phi(g^{-1}) = (-1)^k\phi(g^t)$,
	\begin{equation}\Theta_{q}^{(1)}(\phi) = (-1)^k\int_{G^1/(\Gamma\cap N)} \phi\left(g^t\right) \sum_{m=1}^\infty\hat{\Phi}_q(0,0,0,m)\hat{f}\left(g \cdot \left(0,0,0, \frac{m}{q^2}\right)\right)dg.
	\end{equation}
	Since the point $(0,0,0,1) = y^3$ is invariant under action by $N$, which leaves $y$ fixed, and since there is no constant term in the Fourier expansion of $\phi(g^{t})$ (see \ref{def_cusp_form}), the integral vanishes upon integrating in the $n_x$ variable of the Iwasawa decomposition.
\end{proof}

For $\varepsilon = \pm$ let, for all $\RE(x)$ sufficiently large, 
\begin{equation}
\hat{G}_{\phi}^\varepsilon(x)  = \sum_{\ell,m=1}^\infty \frac{\rho_\phi(\varepsilon 3\ell m)}{\ell^{1+x} (3m)^{1 + 3x}}.
\end{equation}
Given square-free $q$, satisfying $(q,6) = 1$, define for all $\RE(x)>0$ sufficiently large,
\begin{equation}
 \hat{G}_{\phi,q}^\varepsilon(x) = \sum_{\ell,m=1}^\infty \frac{\rho_{\phi}(\varepsilon 3\ell mq)}{\ell^{1+x}(3mq)^{1+3x}}.
\end{equation}
This is a sub-series of the Dirichlet series defining $\hat{G}^\varepsilon_\phi(x)$.

\begin{lemma}\label{G_bound_lemma}
 Let $q$ be square-free satisfying $(q,6)=1$. For $\epsilon > 0$, the functions $\hat{G}_{\phi}^\varepsilon(x)$ and $\hat{G}_{\phi,q}^{\varepsilon}(x)$ are bounded on $\{x: \RE(x) \geq \epsilon\}$ by a constant depending only on $\phi,$ and $\epsilon$.
\end{lemma}
\begin{proof}
 It follows from Rankin-Selberg theory that as $X \to \infty$, 
 \begin{equation}
  \sum_{n \leq X} |\rho_\phi(n)| \ll X.
 \end{equation}
Since the number of ways of writing $n = 3\ell m$ is bounded by $\ll_\epsilon n^{\frac{\epsilon}{2}}$, it follows that
\begin{align*}
 \left|\hat{G}_\phi^\varepsilon(x)\right| &\leq \sum_{\ell, m = 1}^\infty \frac{|\rho_\phi(\varepsilon 3\ell m)|}{|\ell^{1+x} (3m)^{1 + 3x} |}\\
 & \ll_\epsilon \sum_{n = 1}^\infty \frac{|\rho_\phi(n)|}{n^{1 + \frac{\epsilon}{2}}}
\end{align*}
and this sum is bounded by a constant depending only on $\phi$ and $\epsilon$, by partial summation. The same bound applies to $\hat{G}_{\phi,q}^{\varepsilon}(x)$ since it is a sub-series of $\hat{G}_\phi^{\varepsilon}(x)$.
\end{proof}

The Archimedean counterpart to $\hat{G}$ is for $\varepsilon_1, \varepsilon_2  = \pm$,
\begin{align}
 W_{\phi}^{\varepsilon_1, \varepsilon_2}(w_1, w_2) &= \frac{1}{2}\frac{\Gamma(1-w_2)}{(2\pi)^{\frac{1 + w_1 + w_2}{2}}} \left(\cos\left(\frac{\pi}{2}(1-w_2) \right) + i \varepsilon_1\varepsilon_2 \sin\left(\frac{\pi}{2}(1-w_2) \right)\right)\\
 \notag &\times \widetilde{K}_{\varepsilon_1 \frac{k}{2}, it}\left(\frac{w_1 + 3w_2-1}{2}\right).
\end{align}
\begin{lemma}\label{W_bound_lemma}
The function $W_\phi^{\varepsilon_1, \varepsilon_2}(w_1, w_2)$ is holomorphic in $\RE(w_1 + 3w_2)>1$, $\RE(w_2) < 1$. 
 Let $0 < \epsilon < \frac{1}{2}$. For $\epsilon \leq \RE(w_2) \leq 1-\epsilon$, 
 \begin{equation}
  \left|\Gamma(1-w_2) \left(\cos \left(\frac{\pi}{2}(1-w_2) \right) + i\epsilon_1\epsilon_2 \sin \left(\frac{\pi}{2}(1-w_2) \right) \right) \right| \ll |w_2|^{\frac{1}{2}-\RE(w_2)}.
 \end{equation}

\end{lemma}
\begin{proof}
Recall from Lemma \ref{Mellin_transform_lemma} that, to the right of 0, $\tilde{K}_{\varepsilon_1 \frac{k}{2}, it}$ is holomorphic and decays faster than any polynomial in vertical strips.  This suffices to prove the holomorphicity of $W_\phi^{\varepsilon_1, \varepsilon_2}$. 

The claimed bound  follows from Stirling's approximation which gives, in $\RE(z)> \epsilon$,
 \begin{equation}
  \Gamma(z) = \sqrt{\frac{2\pi}{z}} \left(\frac{z}{e} \right)^z \left(1 + O\left( \frac{1}{z}\right)\right).
 \end{equation}
Set $z = (1-w_2) = \sigma + iT$, and assume without loss of generality that $T>1$.  Thus $\log z = \log iT + O\left(\frac{1}{T}\right) = \log T + \frac{i\pi}{2} + O\left(\frac{1}{T} \right)$.  It follows that
\begin{equation}
 \RE(z \log z) = -\frac{\pi}{2} T + \sigma \log T + O(1).
\end{equation}
Since $|e^z|$ is bounded below, it follows that
\begin{equation}
 |\Gamma(z)| \ll \exp\left( -\frac{\pi}{2} T + \left(\sigma-\frac{1}{2}\right) \log T\right).
\end{equation}
The claim follows on considering the exponential growth of $\sin$ and $\cos$ in $T$.
\end{proof}

Set
\begin{equation}
  f_{2\ell}(x) =  \int_0^{1} f(k_{2\pi \theta}\cdot x) e(-2\ell \theta)d\theta.
\end{equation}
Since $k_\theta$ varies in a compact set, if $f$ is Schwarz class, so is $f_{2\ell}$.
Introduce, for $z_1, z_2 \in \bC$
\begin{align}
 \Sigma^{\pm}(f,z_1, z_2) &= \int_0^\infty \int_0^\infty f(0,0,t,\pm u) t^{z_1-1}u^{z_2-1}dt du.
\end{align}
\begin{lemma}
 If $f$ is Schwarz class, then $\Sigma^{\pm}(z_1, z_2)$ is holomorphic in $\RE(z_1), \RE(z_2) > 0$.  In this domain, it satisfies the decay estimate in vertical strips, for $\sigma_1, \sigma_2 > 0$, for any $A_1, A_2 > 0$,
 \begin{equation}
  \left|\Sigma^{\pm}(f, \sigma_1 + it_1, \sigma_2 + it_2) \right| \ll_{A_1, A_2} \frac{1}{(1 + |t_1|)^{A_1} (1 + |t_2|)^{A_2}}.
 \end{equation}
For $t > 0$, if $f^t(x) = f(tx)$, then $\Sigma^{\pm}(f^t, z_1, z_2) = t^{-z_1-z_2} \Sigma^{\pm}(f, z_1, z_2)$.
\end{lemma}
\begin{proof}
 The convergence of the integral in $\RE z_1, \RE z_2 > 0$ is guaranteed since $f$ is Schwarz class, and the holomorphicity follows by differentiating under the integral.  The decay in vertical strips follows on integrating several times by parts.  The dilation follows from changing variables,
 \begin{align}
  \Sigma^{\pm}(f^t, z_1, z_2) &= \int_0^\infty \int_0^\infty f(0,0,tx_1, \pm tx_2) x_1^{z_1-1} x_2^{z_2-1} dx_1 dx_2\\
  \notag &= t^{-z_1-z_2} \int_0^\infty \int_0^\infty f(0,0,x_1,\pm x_2)x_1^{z_1-1}x_2^{z_2-1}dx_1dx_2\\
  \notag &= t^{-z_1-z_2} \Sigma^{\pm}(f, z_1, z_2).
 \end{align}

\end{proof}

The following lemma obtains an expression for $\hat{\Theta}_q^{(2)}$ as a double Mellin transform.

\begin{lemma}\label{contour_expr_lemma}
 For $q$ be square-free with  $(q,6) = 1$, 
  \begin{align}
  &\hat{\Theta}_{  q}^{(2)}( \phi) = (-1)^k\sum_{\varepsilon_1, \varepsilon_2 = \pm}\sum_{q_1q_2 =q}q_1^{-2}\prod_{p|q_2}(p^{-3}-p^{-5})  \\
  \notag&\times \oiint_{\substack{\RE(w_1, w_2) =(1, \frac{1}{2})}} q_2^{2(w_1 +w_2)}\Sigma^{\varepsilon_2}\left(\hat{f}_{2k}, w_1, w_2\right)W_\phi^{\varepsilon_1, \varepsilon_2}(w_1, w_2)\hat{G}_{\phi,q_2}^{\varepsilon_1}\left(\frac{w_1 + w_2 -1}{2} \right)dw_1 dw_2.
 \end{align}

\end{lemma}

\begin{proof}
 We have, using $\phi(g^{-1}) = (-1)^k\phi(g^t)$, and folding together the sum over $\Gamma$ and integral over $G$,
 \begin{align}\label{Theta_integral}
  &\hat{\Theta}_{ q}^{(2)}( \phi)  = (-1)^k\int_{G^1/\Gamma} \phi\left(g^t\right) \sum_{x \in \hat{L}_0(II)}\hat{\Phi}_{q}(x)\hat{f}\left(g \cdot \frac{x}{q^2} \right)dg\\
  \notag &= (-1)^k\int_{G^1/\Gamma}  \phi\left(g^t\right)\sum_{m=1}^\infty \sum_{0 \leq n < 3m} \hat{\Phi}_{q}(0,0,3m,n)\sum_{\gamma \in \Gamma}\hat{f}\left(g \gamma \cdot\left(0,0, \frac{3m}{q^2}, \frac{n}{q^2}\right)\right)dg\\
   \notag &= (-1)^k\int_{G^1}  \phi\left(g^t\right)\sum_{m=1}^\infty \sum_{0 \leq n < 3m} \hat{\Phi}_{q}(0,0,3m,n)\hat{f}\left(g  \cdot\left(0,0, \frac{3m}{q^2}, \frac{n}{q^2}\right)\right)dg.
  \end{align}
 The evaluation of $\hat{\Phi}_q(x)$ from Lemma \ref{FT_eval_lemma} imposes the constraint $q|m$.  The lemma gives
 \begin{align}
  \hat{\Phi}_q(0,0,3mq,n) &= \prod_{p|q, p^2|(mq,n)} (p^{-2} + p^{-3}-p^{-5}) \prod_{p|q, p^2 \nmid (mq,n)} (p^{-3}-p^{-5})\\
  \notag &= \prod_{p|q} (p^{-3}-p^{-5}) \prod_{p|q, p^2 | (mq,n)} \left(1 + \frac{p^{-2}}{p^{-3}-p^{-5}}\right)\\
  \notag &= \sum_{\substack{q_1q_2 = q, \\q_1^2 |(mq, n)}} q_1^{-2}\prod_{p|q_2} (p^{-3}-p^{-5}).
 \end{align}
Replacing $m=: m q q_1 $ and $n =: n q_1^2 $ in (\ref{Theta_integral})  obtains
 \begin{align}
  &\hat{\Theta}_{ q}^{(2)}( \phi) = (-1)^k\sum_{q_1q_2 = q}q_1^{-2}\prod_{p|q_2}(p^{-3}-p^{-5})\\
 \notag &\quad \times \int_{G^1}  \phi\left(g^t\right)\sum_{m=1}^\infty \sum_{0 \leq n < 3mq_2} \hat{f}\left(g\cdot \left(0,0, \frac{3m}{q_2}, \frac{n}{q_2^2}\right)\right) dg.
 \end{align}
 Write $g = k_\theta a_t n_u$ and integrate in $\theta$.  Since $\phi(g^t)$ transforms under $k_\theta^t$ by a character of degree $-2k$, the integral replaces $\hat{f}$ with $\hat{f}_{2k}$.  Since $n_u$ maps $xy^2 \mapsto xy^2 + uy^3$,  
 \begin{align}
 \notag &\hat{\Theta}_{ q}^{(2)}( \phi)= (-1)^k\sum_{q_1q_2 = q}q_1^{-2}\prod_{p|q_2}(p^{-3}-p^{-5})\int_0^\infty t^{-3}dt \int_{-\infty}^\infty du\\
 \notag &\quad \times  \sum_{m=1}^\infty \sum_{0 \leq n < 3mq_2}^\infty  \hat{f}_{2k}\left(a_t\cdot \left(0,0, \frac{3m}{q_2}, \frac{n + 3mq_2u}{q_2^2} \right)\right)\phi((a_t n_u)^t).
 \end{align} 
 Making a linear change of variable in $u$ obtains
 \begin{align}
  &\hat{\Theta}_{ q}^{(2)}( \phi) = (-1)^k\sum_{q_1q_2 = q}q_1^{-2}\prod_{p|q_2}(p^{-3}-p^{-5})\\
  \notag &\times \int_0^\infty t^{-3}dt \int_{-\infty}^\infty du \sum_{m=1}^\infty \hat{f}_{2k}\left(a_t \cdot \left(0,0,\frac{3m}{q_2}, \frac{u}{q_2^2} \right) \right)\frac{1}{3mq_2}\sum_{0 \leq n < 3mq_2}\phi\left((a_t n_{\frac{u-n}{3mq_2}})^t\right).
 \end{align}
Expand $\phi$ in Fourier series, abbreviating here and in what follows $\check{K}_{\frac{\varepsilon k}{2}, s-\frac{1}{2}} = \check{K}_\varepsilon$,\footnote{Note that  $n_u^t a_t$ corresponds to $u + it^2$ in the upper half plane model.}
\begin{equation}
 \phi\left(\left(a_t n_{\frac{u-n}{3mq_2}}\right)^t\right) = t\sum_{\varepsilon = \pm} \sum_{\ell = 1}^\infty \rho_\phi(\varepsilon \ell) \check{K}_\varepsilon (2\pi \ell t^2)e\left(\varepsilon \ell \frac{u-n}{3mq_2}\right). 
\end{equation}
  The sum over $n$ selects Fourier coefficients with frequencies $\ell$ divisible by $3mq_2$.  Replacing $\ell$ with $ 3mq_2\ell := \ell$,
\begin{align}
 \hat{\Theta}_{ q}^{(2)}( \phi) =& (-1)^k\sum_{\varepsilon = \pm}\sum_{q_1q_2=q} q_1^{-2}\prod_{p|q_2}(p^{-3}-p^{-5})  \int_0^\infty \int_{-\infty}^\infty \\\notag
&\times\sum_{\ell, m=1}^\infty \rho_{\phi}(\varepsilon 3\ell mq_2) \check{K}_\varepsilon(6\pi \ell m q_2 t^2)\hat{f}_{2k}(0,0,3t^{-1}q_2^{-1}m, u q_2^{-2})e(\varepsilon \ell t^3 u) du tdt.
\end{align}
Here one factor of $t$ has been gained from the Fourier expansion of $\phi$ and a factor of $t^3$ was gained by replacing $ut^{-3}$ with $u$.

Take Mellin transforms in both variables in $\hat{f}_{2k}$, writing $u = \varepsilon_2 |u|$, to obtain
\begin{align}
  &\hat{\Theta}_{ q}^{(2)}( \phi) =(-1)^k\sum_{\varepsilon_1, \varepsilon_2 = \pm}\sum_{q_1q_2 = q}q_1^{-2}\prod_{p|q_2}(p^{-3}-p^{-5}) \\ \notag &\times \int_0^\infty \int_{0}^\infty \sum_{\ell, m=1}^\infty \rho_{\phi}(\varepsilon_1 3 \ell mq_2)
 \check{K}_{\varepsilon_1}(6\pi \ell m q_2 t^2)\\\notag&\times\oiint_{\RE(w_1, w_2) = (1, \frac{1}{2})} \Sigma^{\varepsilon_2}\left(\hat{f}_{2k}, w_1, w_2\right) \left(\frac{tq_2}{3m}\right)^{w_1}\left(\frac{q_2^2}{u}\right)^{w_2}e(\varepsilon_1 \varepsilon_2 \ell t^3 u)dw_1dw_2 du tdt.
\end{align}
Note that since $\hat{f}_{2k}$ is Schwarz class, the Mellin transform $\Sigma^{\varepsilon_2}\left(\hat{f}_{2k}, w_1, w_2\right)$ decays faster than any polynomial in vertical strips.  This justifies exchange in the order of integration.

Make the change of variables $u:= 2\pi \ell t^3 u$, which replaces 
\begin{equation}
 u^{-w_2}du := (2\pi \ell t^3)^{w_2-1}u^{-w_2}du,
\end{equation}
 then $t:= 6\pi \ell m q_2 t^2$, which replaces 
\begin{equation}
 t^{w_1 + 3w_2 -2} dt := \frac{1}{2} (6\pi \ell mq_2)^{\frac{1- w_1-3w_2}{2}} t^{\frac{w_1 +3w_2 -3}{2}}dt.
\end{equation}
Thus
\begin{align}
 &\hat{\Theta}_{ q}^{(2)}( \phi) =(-1)^k\sum_{\varepsilon_1, \varepsilon_2 = \pm}\sum_{q_1q_2 = q} q_1^{-2}\prod_{p|q_2}(p^{-3}-p^{-5}) \\ \notag&\times \oiint_{\RE(w_1,w_2) = (1, \frac{1}{2})}q_2^{2(w_1+w_2)}\Sigma^{\varepsilon_2}\left(\hat{f}_{2k}, w_1, w_2\right) \sum_{\ell, m=1}^\infty \frac{\rho_{\phi}(\varepsilon_1 3 \ell mq_2)}{\ell^{\frac{1+w_1 + w_2 }{2}}(3mq_2)^{\frac{-1+3(w_1+w_2)}{2}}}\\&\notag\times \frac{1}{2}\frac{1}{(2\pi)^{\frac{1+w_1 + w_2}{2}}}  \int_0^\infty e^{i \varepsilon_1\varepsilon_2 u} u^{-w_2}du \int_0^\infty\check{K}_{\varepsilon_1}(t)t^{\frac{w_1+3w_2-3}{2}}dtdw_1dw_2.
\end{align}
The Dirichlet series evaluates to $\hat{G}_{\phi, q_2}^{\varepsilon_1}\left(\frac{w_1+w_2-1}{2}\right)$, while the Mellin transforms in the last line combine with the other factors to give $W_\phi^{\varepsilon_1, \varepsilon_2}(w_1,w_2)$, so that 
\begin{align}
& \hat{\Theta}_{ q}^{(2)}( \phi)=(-1)^k \sum_{\varepsilon_1, \varepsilon_2 = \pm} \sum_{q_1q_2 =q }q_1^{-2}\prod_{p|q_2}(p^{-3}-p^{-5})  \\ \notag
&  \times \oiint_{\RE(w_1,w_2) = (1, \frac{1}{2})} q_2^{2(w_1+w_2)}\Sigma^{\varepsilon_2}\left(\hat{f}_{2k},w_1,w_2\right)W_\phi^{\varepsilon_1, \varepsilon_2}(w_1,w_2)\hat{G}_{\phi, q_2}^{\varepsilon_1}\left(\frac{w_1+w_2-1}{2}\right)dw_1dw_2.
\end{align}

\end{proof}
We can now evaluate $Z^{\pm, 0}_q(\hat{f}, \hat{L}, \phi; s)$.
\begin{lemma}\label{Z_0_lemma}
 The singular integral has the evaluation
 \begin{align}
  &Z^{\pm, 0}_q(\hat{f}, \hat{L}, \phi; s)\\\notag&=(-1)^k \sum_{\varepsilon_1, \varepsilon_2 = \pm}\sum_{q_1q_2 = q}q_1^{-2}\prod_{p|q_2}(p^{-3}-p^{-5})   \oiint_{\RE(w_1,w_2) = (1, \frac{1}{2})} \frac{q_2^{2(w_1+w_2)}}{12s-12+3w_1+3w_2}\\ \notag
& \qquad \times \Sigma^{\varepsilon_2}(\hat{f}_{2k},w_1,w_2)W_\phi^{\varepsilon_1, \varepsilon_2}(w_1,w_2)\hat{G}_{\phi, q_2}^{\varepsilon_1}\left(\frac{w_1+w_2-1}{2}\right)dw_1dw_2
 \end{align}
 and has holomorphic continuation to $\bC$.
\end{lemma}
\begin{proof}
Note that if $f^t$ denotes the dilation by $t$ on $V_{\bR}$, $f^t(x) = f(tx)$, then $\Sigma^\epsilon(f^t, w_1,w_2) = t^{-w_1-w_2}\Sigma^\epsilon(f, w_1,w_2)$. Also, acting by the scalar matrix $\begin{pmatrix} \lambda &0\\ 0&\lambda\end{pmatrix}$ scales the form $x$ by $\lambda^3$. Thus
\begin{align}
 Z^{\pm, 0}_q(\hat{f}, \hat{L}, \phi; s) &= \int_{\substack{G^+/\Gamma\\ \chi(g) \leq 1}} \chi(g)^{s-1} \phi(g^{-1})\sum_{x \in \hat{L}_0} \hat{\Phi}_q(x) \hat{f}\left(g^{\iota}\cdot\frac{x}{q^2} \right) dg\\
\notag &= \int_0^1 \lambda^{12(s-1)} \sI\left(\hat{f}^{\frac{1}{\lambda^3}}, \phi\right) \frac{d\lambda}{\lambda}.
\end{align}
Since 
 $\sI\left(\hat{f}, \phi\right) = \hat{\Theta}_{ q}^{(2)}( \phi)$,
 \begin{align}
  &Z^{\pm, 0}_q(\hat{f}, \hat{L}, \phi; s)=(-1)^k\sum_{\varepsilon_1, \varepsilon_2 = \pm}\sum_{q_1q_2 = q}q_1^{-2} \prod_{p|q_2}(p^{-3}-p^{-5})    \oiint_{\RE(w_1,w_2) = (1, \frac{1}{2})} q_2^{2(w_1+w_2)}\\ \notag
& \qquad \times \Sigma^{\varepsilon_2}(\hat{f}_{2k},w_1,w_2)W_\phi^{\varepsilon_1, \varepsilon_2}(w_1,w_2)\hat{G}_{\phi, q_2}^{\varepsilon_1}\left(\frac{w_1+w_2-1}{2}\right)\int_0^1\lambda^{12s-12+3w_1+3w_2}\frac{d\lambda}{\lambda}dw_1dw_2.
\end{align}
The evaluation follows on integrating in $\lambda$.
Shifting the $w_1$ contour rightward obtains the holomorphic continuation of $Z_q^{\pm, 0}$ in $s$.
\end{proof}
\section{Proof of Theorem \ref{cubic_field_theorem}}
We now give the proof of Theorem \ref{cubic_field_theorem}.

Let $F$ be the smooth function of the theorem and define
\begin{equation}\label{weighted_sum}
 N_{3, \pm}'(\phi, F,X) = \sum_{m \in \zed \setminus \{0\}} F\left(\frac{\pm m}{X}\right){\sum_{i = 1,*}^{h(m)}}\frac{\phi(g_{i,m})}{|\Gamma(i,m)|}
\end{equation}
with the $*$ restricting summation to classes which are maximal at all primes $p$.
\begin{lemma}\label{field_lemma} Let $\phi$ be a cusp-form.
 The count of fields from Theorem \ref{cubic_field_theorem} satisfies
 \begin{equation}
  N_{3, \pm}(\phi, F, X) = \frac{1}{2}N_{3, \pm}'(\phi, F, X) + O\left(\|\phi\|_\infty X^{\frac{1}{2}}\right).
 \end{equation}

\end{lemma}
\begin{proof}
 By the Delone-Faddeev correspondence the sum in (\ref{weighted_sum}) counts fields of degree at most 3.  $S_3$ cubic fields are counted with weight 2, while cyclic cubic fields are counted with weight $\frac{2}{3}$, quadratic fields are counted with weight 1 and $\bQ$ is counted with weight $\frac{1}{3}$, see Proposition 5.1 of \cite{TT13b}, or \cite{BST13} for a detailed discussion.
 
 There are $O\left(X^{\frac{1}{2}}\right)$ cyclic cubic extensions of discriminant at most $X$, which accounts for the error term.  According as the discriminant is 0 or 1 modulo 4, the quadratic fields are represented by forms 
 $(-D/4, 0, 1, 0)$ or $(-(D-1)/4, 1, 1, 0)$ which have an irreducible quadratic factor of the corresponding discriminant.  Since $a_t$ acts on $x_{\pm}$ with $a_t \cdot x_{\pm} = \left(\frac{t^3}{\sqrt{2}}, 0, \mp \frac{1}{t\sqrt{2}}, 0\right)$, solving $g \cdot x_{\pm} = (-D/4, 0, 1, 0)$ has $t(g)$ of order $D^{\frac{1}{4}}$, and the case of $(-(D-1)/4, 1, 1, 0)$ is similar.  Thus the sum over these fields is $O(\|\phi\|_\infty)$ by the exponential decay of the cusp-form $\phi$. The contribution from $\bQ$ is $O\left(\|\phi\|_\infty\right)$.
\end{proof}
Going forward we handle just the sum over positive $m$, the negative part being treated similarly. 
By inclusion-exclusion,
\begin{equation}
 N_3'(\phi, F,X) = \sum_{q} \mu(q)\sum_{m =1}^{\infty} F\left(\frac{m}{X}\right){\sum_{i = 1}^{h(m)}}\Phi_q(x_{i,m})\frac{\phi(g_{i,m})}{|\Gamma(i,m)|}.
\end{equation}

\begin{lemma}\label{tail_lemma}
 The tail of the sieve satisfies the bound
 \begin{equation}
  \sum_{q > Q} \mu(q)\sum_{m \in \zed \setminus \{0\}} F\left(\frac{m}{X}\right){\sum_{i = 1}^{h(m)}}\Phi_q(x_{i,m})\frac{\phi(g_{i,m})}{|\Gamma(i,m)|} \ll_\epsilon \|\phi\|_\infty \frac{X}{Q^{1-\epsilon}}.
 \end{equation}

\end{lemma}

\begin{proof}
 Lemma 3.4 of \cite{TT13b} proves that for square-free $r$, and $Y > 1$,
 \begin{equation}
  \sum_{\substack{|\Disc(x)| < Y\\ r^2| \Disc(x)}} 1 < \frac{M 6^{\omega(r)}Y}{r^2}
 \end{equation}
where the sum is over non-singular binary cubic forms up to $\GL_2(\zed)$ equivalence, where $M$ is an absolute constant, and where $\omega(r)$ is the number of prime factors of $r$. Since for  square-free $r$, an index $r$ subring of a maximal ring has discriminant divisible by $r^2$, those forms selected by $\Phi_q$ have discriminant divisible by $q^2$.  Hence,  for any $Q > 1$, for all $\epsilon > 0$,
\begin{align}\label{integral_representation}
 &\sum_{q > Q} \mu(q)\sum_{m \in \zed \setminus \{0\}} F\left(\frac{m}{X}\right){\sum_{i = 1}^{h(m)}}\Phi_q(x_{i,m})\frac{\phi(g_{i,m})}{|\Gamma(i,m)|} \\
 \notag & \ll \|\phi\|_\infty \sum_{q > Q} \mu^2(q) \frac{X 6^{\omega(q)}}{q^2} \ll_\epsilon \|\phi\|_\infty \frac{X}{Q^{1-\epsilon}}.
\end{align}
\end{proof}

It remains to  control the main part of the sieve.
 Write, in $\RE(s) > 0$,
\begin{equation}
\tilde{F}(s) = \int_0^\infty F(x)x^{s-1}dx
\end{equation}
for the Mellin transform. Since $F$ is smooth and of compact support, $\tilde{F}(s)$ is entire.  
Set $f_D$ by twice applying the operator $x \frac{d}{dx}$ to $F$, so that $\tilde{f}_D(s) = s^2 \tilde{F}(s)$, see (\ref{mellin_operator}).   Also, choose $f_G$ such that $\Lambda_{f_G, \phi} \neq 0$.

The part of the sum in $q \leq Q$ may be expressed, by Mellin inversion, as
\begin{align}
 N_3''(\phi, F, X) := \sum_{q \leq Q} \mu(q) \int_{\RE(s) = 2} \tilde{F}(s) X^s \sL_q^+(\phi,s) ds.
\end{align}
Note that on this line, $\sL_q^+(\phi, s)$ is defined by an absolutely convergent Dirichlet series, so that the convergence follows from the rapid decay of $\tilde{F}(s)$.
  Write, where $\tilde{f}_D(s) \neq 0$,
\begin{align}
 \sL_q^+(s,\phi) &= \frac{12 Z^+_q(f,\phi,L;s)}{ \Lambda_{f_G,\phi} \tilde{f}_D(s)} \\\notag&= \frac{12}{ \Lambda_{f_G, \phi}\tilde{f}_D(s)}\left(Z_q^{+,+}(f,\phi,L;s) + \hat{Z}_q^{+,+}(\hat{f},\phi,\hat{L};1-s) + Z_q^{+,0}(\hat{f},\phi,\hat{L};s) \right).
\end{align}
Thus
\begin{equation}
 N_3''(\phi, F, X) := \frac{12}{ \Lambda_{f_G,\phi}}\sum_{q \leq Q} \mu(q) \int_{\RE(s) = 2}  X^s Z^+_q(f,\phi,L;s) \frac{ds}{s^2}.
\end{equation}

The contour integral (\ref{integral_representation}) is the sum of three terms, corresponding to $Z_q^{+,+}$, $\hat{Z}_q^{+,+}$ and $Z_q^{+,0}$. Write these terms as $N_3^{+,+}$, $\hat{N}_3^{+,+}$ and $N_3^{+,0}$, which are bounded separately. 

\begin{lemma}\label{lemma_pp}
 We have the bound, for any $\epsilon > 0$,
 \begin{equation}
  N_3^{+,+}(\phi, F, X)= \frac{12}{ \Lambda_{f_G,\phi}}\sum_{q \leq Q} \mu(q) \int_{\RE(s) = 2} X^s Z_q^{+,+}(f,\phi,L;s)\frac{ds}{s^2} \ll_{\phi,\epsilon} X^\epsilon.
 \end{equation}

\end{lemma}
\begin{proof}
 Note that $\Phi_q(x) \neq 0$ implies that $q|\Disc(x)$.  Shift the integral to $\RE(s) = \epsilon$ and open the definition of the orbital integral to obtain
 \begin{align}
 &N_3^{+,+}(\phi, F, X)\\ \notag &= \frac{12}{ \Lambda_{f_G,\phi}}\sum_{q \leq Q} \mu(q) \int_{\RE(s) = \epsilon}  X^s \int_{G^+/\Gamma, \chi(g)\geq 1} \chi(g)^s \phi\left(g^{-1}\right) \sum_{x \in L} \Phi_q(x)f(g\cdot x) dg \frac{ds}{s^2}\\
  \notag &=\frac{12}{ \Lambda_{f_G,\phi}}\sum_{q \leq Q} \mu(q) \int_{\RE(s) = \epsilon}  X^s \sum_{m=1}^\infty \sum_{i=1}^{h(m)} \frac{\Phi_q(x_{i,m})}{|\Gamma(i,m)|}\\\notag&\times\int_{g \in G^+, \chi(g) \geq 1} \chi(g)^s \phi(g^{-1})f(g \cdot x_{i,m})dg\frac{ds}{s^2}.
 \end{align}
The sum over $m$ is finite due to the compact support of $f$, hence the sum over $q$ is bounded, also.  Note that this also justifies the convergence of the contour integrals.
\end{proof}
\begin{lemma}\label{lemma_hat_pp}
 We have the bound, for any $\epsilon > 0$,
 \begin{equation}
 \hat{N}_3^{+,+}(\phi, F, X)=  \frac{12}{ \Lambda_{f_G,\phi}}\sum_{q \leq Q} \mu(q) \int_{\RE(s) = 2}  X^s \hat{Z}_q^{+,+}(\hat{f},\phi,\hat{L};1-s)\frac{ds}{s^2} \ll_{\phi,\epsilon} Q^2X^\epsilon.
 \end{equation}

\end{lemma}
\begin{proof}
 Shift the contour to $\RE(s) = \epsilon$ and open the orbital integral to obtain
 \begin{align}&\hat{N}_3^{+,+}(\phi, F, X)\\\notag&= \frac{12}{\Lambda_{f_G, \phi}} \sum_{q \leq Q} \mu(q) \oint_{\RE(s) = \epsilon} X^s \int_{G^+/\Gamma, \chi(g)\geq 1} \chi(g)^{1-s}\phi\left(g^{-1} \right)\sum_{x \in \hat{L} \setminus \hat{L}_0} \hat{\Phi}_q(x) \hat{f}\left(g \cdot \frac{x}{q^2} \right)dg \frac{ds}{s^2} 
\\ \notag &= \frac{12}{\Lambda_{f_G, \phi}} \sum_{q \leq Q} \mu(q) \oint_{\RE(s) = \epsilon} X^s \int_{g \in G^+, \chi(g) \geq 1} \chi(g)^{1-s}\phi\left(g^{-1}\right) \\& \notag \times\sum_{ m \neq 0} \sum_{i=1}^{\hat{h}(m)} \frac{\hat{\Phi}_q\left(\hat{x}_{i,m}\right)}{\left|\hat{\Gamma}(i,m)\right|}\hat{f}\left(g \cdot \frac{\hat{x}_{i,m}}{q^2}\right) dg \frac{ds}{s^2}.
\end{align} 
Make the change of variables $g := \begin{pmatrix} q^{-\frac{2}{3}} & \\ &q^{-\frac{2}{3}}\end{pmatrix}g \hat{g}_{i,m}$ to write this as
\begin{align}
 \hat{N}_3^{+,+}(\phi, F, X)&=  \frac{12}{ \Lambda_{f_G,\phi}}\sum_{q \leq Q} \mu(q) \oint_{\RE(s) = \epsilon} X^sq^{8(1-s)} \sum_{m \neq 0} \sum_{i=1}^{\hat{h}(m)} \frac{\hat{\Phi}_q\left(\hat{x}_{i,m}\right)}{\left|\hat{\Gamma}(i,m)\right|} \frac{1}{|m|^{1-s}}\\\notag &\times\int_{g \in G^+, \chi(g) \geq \frac{|m|}{q^8}} \chi(g)^{1-s}\phi\left(\hat{g}_{i,m}g^{-1}\right)\hat{f}(g\cdot x_{\sgn m}) dg \frac{ds}{s^2}.
 \end{align}
The integral over $G^+$ and the rapid decay of $\hat{f}$ effectively limit summation over $m$ to $|m|\ll q^{8}X^\epsilon$.  This also justifies the convergence in the contour integrals.  Bound the integral over $g$ by a constant depending on $f$ and $\phi$, and bound the sum over $m$ by  applying the bound of Theorem \ref{FT_bound_theorem},
\begin{equation}
 \sum_{|m| \leq Y} \sum_{i=1}^{\hat{h}(m)} \frac{\left|\hat{\Phi}_q\left(\hat{x}_{i,m}\right)\right|}{\left|\hat{\Gamma}(i,m)\right|} \ll_\epsilon Y q^{-7+\epsilon}
\end{equation}
with $Y \ll q^{8}X^\epsilon$. Apply partial summation and sum in $q \leq Q$ to obtain the lemma.  
\end{proof}

\begin{lemma}\label{lemma_p0}
 We have the bound, for all $\epsilon>0$,
 \begin{equation}
 N_3^{+,0}(\phi, F, X)=  \frac{12}{ \Lambda_{f_G,\phi}}\sum_{q \leq Q} \mu(q) \int_{\RE(s) = 2}  X^s Z_q^{+,0}(\hat{f},\phi,\hat{L};s)\frac{ds}{s^2} \ll_{\phi,\epsilon} X^{\frac{1}{4}}(QX)^\epsilon.
 \end{equation}

\end{lemma}
\begin{proof} For $q$ co-prime to 2 and 3, in the integral representation
 \begin{align}
  Z^{+, 0}_q(\hat{f}, \phi, \hat{L}; s) &= (-1)^k\sum_{\varepsilon_1, \varepsilon_2 = \pm}\sum_{q_1q_2 = q}q_1^{-2}\prod_{p|q_2}(p^{-3}-p^{-5}) \\&\notag\qquad \times \oiint_{\RE(w_1,w_2) = (1, \frac{1}{2})} \frac{q_2^{2(w_1+w_2)}}{12s-12+3w_1+3w_2}\\ \notag
& \qquad \times \Sigma^{\varepsilon_2}\left(\hat{f}_{2k},w_1,w_2\right)W_\phi^{\varepsilon_1, \varepsilon_2}(w_1,w_2)\hat{G}_{\phi,q_2}^{\varepsilon_1}\left(\frac{w_1+w_2-1}{2}\right)dw_1dw_2
 \end{align}
shift the $w_1$ contour to $\RE(w_1) = \frac{5}{2}-\epsilon$ so that $\RE(w_1 + w_2) = 3-\epsilon$. By Lemmas \ref{G_bound_lemma} and \ref{W_bound_lemma}, on these lines \begin{equation}W_\phi^{\varepsilon_1, \varepsilon_2}(w_1,w_2)\hat{G}_{\phi, q_2}^{\varepsilon_1}\left(\frac{w_1+w_2-1}{2}\right)\end{equation} is uniformly bounded, so that the rapid decay of $\Sigma^{\varepsilon_2}$ guarantees convergence.  The $s$ contour may now be shifted to $\RE(s) = \frac{1}{4} + \epsilon$, where the factor $\frac{1}{12s-12 + 3w_1 +3w_2}$ is bounded by a quantity depending only on $\epsilon$.

Introduce the sum over $q$, 
\begin{align*}
& \sum_{q \leq Q, (q,6)=1} \mu(q) \sum_{q_1q_2 = q}q_1^{-2} q_2^{2(w_1 + w_2)} \prod_{p|q_2} (p^{-3}-p^{-5})  \hat{G}^{\varepsilon_1}_{\phi,q_2}\left(\frac{w_1+w_2-1}{2} \right)\\
&= \sum_{\ell, m \geq 1} \sum_{ q \leq Q, (q,6)=1} \mu(q) \sum_{q_1q_2 = q}q_1^{-2} q_2^{2(w_1 + w_2)} \prod_{p|q_2} (p^{-3}-p^{-5})  \frac{\rho_\phi(3\ell m q_2)}{\ell^{\frac{1 + w_1 + w_2}{2}} (3mq_2)^{\frac{3(w_1 + w_2)-1}{2}}}\\
&= \sum_n \rho_\phi(3n)c_n(w_1,w_2). 
\end{align*}
The number $c_n(w_1,w_2)$ is a sum over $ q_2 \ell m = n$ multiplied by \begin{equation}\sum_{q_1 \leq \frac{Q}{q_2}, (q_1, 6)=1} \mu(q_1)q_1^{-2},\end{equation} which is a quantity which is bounded.  In the sum over $q_2, \ell, m$ the dependence on $q_2$ is bounded by $q_2^{-1 -\frac{\epsilon}{2}}$, and the remaining terms are bounded by a larger negative power.  Hence $|c_n(w_1,w_2)| \ll_\epsilon \frac{d_3(n)}{n^{1 + \frac{\epsilon}{2}}}$ where $d_3(n)$ is the 3-divisor function. Since $d_3(n) \ll_\epsilon n^{\frac{\epsilon}{4}}$, it follows that $|c_n(w_1,w_2)| \ll_\epsilon \frac{1}{n^{1 + \frac{\epsilon}{4}}}$, uniformly in $w_1$ and $w_2$ on the lines of integration.  It now follows from the Rankin-Selberg estimate $\sum_{n \leq X} |\rho_\phi(n)|^2 \ll_\phi X$ as $X \to \infty$  that $\sum_n |\rho_\phi(3n)c_n(w_1,w_2)|$ is uniformly bounded by a quantity depending only on $\phi$ and $\epsilon$.
The argument to this point has treated $q$ co-prime to 2 and 3.  Handling $q$ with one or both of these factors can be done by including the appropriate $\hat{\Phi}_p(x)$ is Lemma \ref{contour_expr_lemma}, and tracking the change to Lemma \ref{Z_0_lemma}.  As maximality modulo 2 and 3 is defined modulo 36, this makes only a bounded change, the details are left to the reader.

On the lines of integration, $X^s$ is bounded by $X^{\frac{1}{4} + \epsilon}$. The claim follows since the integrals are convergent.
\end{proof}

\begin{proof}[Proof of Theorem \ref{cubic_field_theorem}]
 By Lemma \ref{field_lemma},
 \begin{equation}
    N_{3, \pm}(\phi, F, X) = \frac{1}{2}N_{3, \pm}'(\phi, F, X) + O\left(\|\phi\|_\infty X^{\frac{1}{2}}\right).
 \end{equation}
 Truncating the tail using Lemma \ref{tail_lemma} obtains
 \begin{equation}
  N_{3, \pm}'(\phi, F, X) = N_{3, \pm}''(\phi, F, X) + O_\epsilon\left(\|\phi\|_\infty \frac{X}{Q^{1-\epsilon}} \right).
 \end{equation}
 Combining Lemmas \ref{lemma_pp}, \ref{lemma_hat_pp} and \ref{lemma_p0} obtains
 \begin{equation}
  N_{3, \pm}''(\phi, F, X) \ll_{\phi, \epsilon} X^{\frac{1}{4}}(QX)^{\epsilon} + Q^2X^\epsilon.
 \end{equation}
 Choosing $Q = X^{\frac{1}{3}}$ optimizes the error terms.  
\end{proof}

\section{Acknowledgements}
The author thanks the referees for a careful reading of the manuscript and helpful comments.

\bibliographystyle{plain}

\end{document}